\documentclass{amsart}
\usepackage{amsmath,amsthm,amssymb,amsfonts}
\usepackage[utf8]{inputenc}
\usepackage{enumitem}

\usepackage{tikz}
\usetikzlibrary{patterns}
\usetikzlibrary{calc}

\usepackage{xcolor}

\usepackage{aliascnt}
\usepackage[unicode,bookmarksopen]{hyperref}
\newtheorem{theorem}{Theorem}[section]
\newcommand{\MakeTheoremAndCounter}[2]{\newaliascnt{#1}{theorem}
\newtheorem{#1}[#1]{#2}
\aliascntresetthe{#1}
\expandafter\providecommand\csname#1autorefname\endcsname{#2}}

\theoremstyle{plain}
\MakeTheoremAndCounter{lemma}{Lemma}
\MakeTheoremAndCounter{claim}{Claim}
\MakeTheoremAndCounter{corollary}{Corollary}
\MakeTheoremAndCounter{proposition}{Proposition}
\MakeTheoremAndCounter{conjecture}{Conjecture}
\MakeTheoremAndCounter{question}{Question}
\MakeTheoremAndCounter{fact}{Fact}

\newcommand{\declarerestatehelper}[2]{\theoremstyle{plain}\newtheorem*{#1}{\autoref{#2}}\expandafter\providecommand\csname#1autorefname\endcsname{\autoref*{#2}}}
\newcounter{restateidx}
\newenvironment{restate}[1]{\declarerestatehelper{autorestate\therestateidx}{#1} \begin{autorestate\therestateidx}}{\end{autorestate\therestateidx}\stepcounter{restateidx}}

\numberwithin{equation}{section}

\theoremstyle{definition}
\MakeTheoremAndCounter{definition}{Definition}
\MakeTheoremAndCounter{example}{Example}
\MakeTheoremAndCounter{remark}{Remark}

\DeclareMathOperator{\pr}{pr}
\DeclareMathOperator{\diam}{diam}

\newcommand{\diaminfty}{\diam^\infty}
\DeclareMathOperator{\interior}{int}
\DeclareMathOperator{\length}{length}
\DeclareMathOperator{\lt}{lt}
\DeclareMathOperator{\LT}{LT}
\DeclareMathOperator{\slt}{slt}
\DeclareMathOperator{\refax}{refc}
\newcommand{\R}{\mathbb{R}}
\newcommand{\N}{\mathbb{N}}
\newcommand{\Z}{\mathbb{Z}}

\newcommand{\coords}[1]{(0, 0), (1, 0), (2, 1), (3, 0), (3, 2), (4, 3), (5, 3), (6, 3), (7, 3),#1 (8, 3), (9, 3), (9, 4), (9, 5), (8, 6), (8, 7), (8, 8), (7, 9), (9, 9),#1 (6, 0), (7, 0), (8, 0), (9, 0), (0, 9), (1, 9), (2, 8), (3, 7), (4, 7)}
\edef\plaincoords{\coords{}}
\newcounter{coordcount}
\foreach \c in \plaincoords {\stepcounter{coordcount}}

\newcommand{\squaregrid}[1][1]{
\draw[very thin, dotted] (0,0) grid[step=1] (#10,#10);
\node[anchor=east] at (0,0) {$(0,0)$};
\node[anchor=east] at (0,#10) {$(0,#1)$};
\node[anchor=west] at (#10,0) {$(#1,0)$};
\node[anchor=west] at (#10,#10) {$(#1,#1)$};
}

\title{Compact sets with large projections and nowhere dense sumset}

\author{Rich\'ard Balka}
\address{Alfr\'ed R\'enyi Institute of Mathematics, Re\'altanoda u.~13--15, H-1053 Budapest, Hungary}
\email{balka.richard@renyi.hu}

\author{M\'arton Elekes}
\address{Alfr\'ed R\'enyi Institute of Mathematics, Re\'altanoda u.~13--15, H-1053 Budapest, Hungary AND E\"otv\"os Lor\'and University, Institute of Mathematics, P\'azm\'any P\'eter s.~1/c, 1117 Budapest, Hungary}
\email{elekes.marton@renyi.hu}
\urladdr{http://www.renyi.hu/$\sim$emarci}

\author{Viktor Kiss}
\address{Alfr\'ed R\'enyi Institute of Mathematics, Re\'altanoda u.~13--15, H-1053 Budapest, Hungary}
\email{kiss.viktor@renyi.hu}

\author{Don\'at Nagy}
\address{E\"otv\"os Lor\'and University, Institute of Mathematics, P\'azm\'any P\'eter s.~1/c, 1117 Budapest, Hungary}
\email{m1nagdon@gmail.com}

\author{M\'ark Po\'or}
\address{E\"otv\"os Lor\'and University, Institute of Mathematics, P\'azm\'any P\'eter s.~1/c, 1117 Budapest, Hungary AND Einstein Institute of Mathematics, Edmond J.~Safra Campus, Givat Ram, The Hebrew University of Jerusalem, Jerusalem, 9190401, Israel}
\email{sokmark@gmail.com}

\thanks{The authors were supported by the National Research, Development and Innovation Office -- NKFIH, grants no.~113047, 129211 and 124749. The first author was supported by the MTA Premium Postdoctoral Research Program. The third author was supported by the National Research, Development and Innovation Office -- NKFIH, grant no.~128273.\\
\includegraphics[height=1cm]{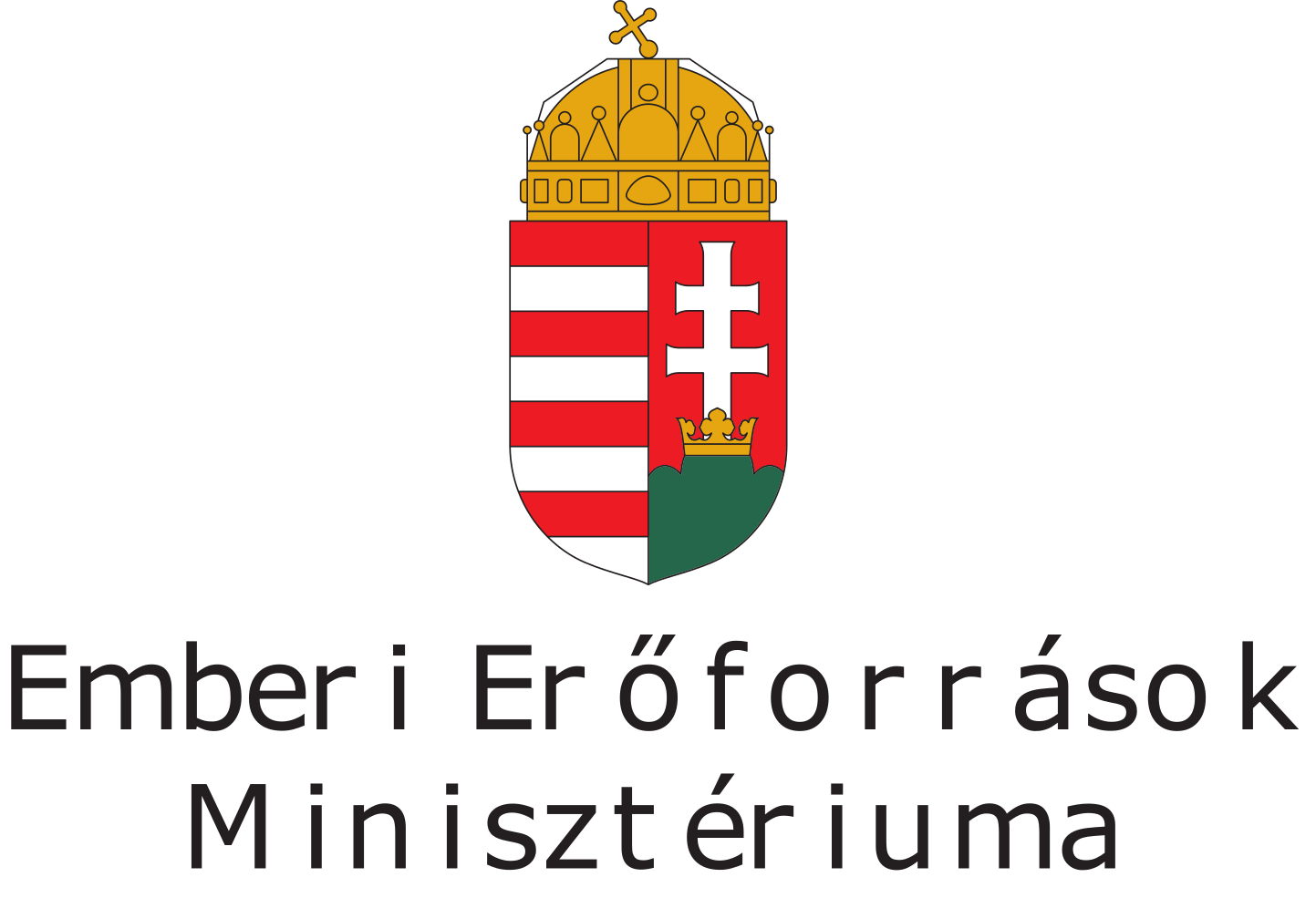} \raisebox{0.5cm}{\parbox[c]{11cm}{The fifth author was supported through the New National Excellence Program of the Ministry of Human Capacities.}}}



\begin{document}

\begin{abstract}
We answer a question of Banakh, Jab\l{}o\'nska and Jab\l{}o\'nski by showing that for $d\ge 2$ there exists a compact set $K \subseteq \mathbb{R}^d$ such that the projection of $K$ onto each hyperplane is of non-empty interior, but $K+K$ is nowhere dense. The proof relies on a random construction.

A natural approach in the proofs is to construct such a $K$ in the unit cube with full projections, that is, such that the projections of $K$ agree with that of the unit cube. We investigate the generalization of these problems for projections onto various dimensional subspaces as well as for $\ell$-fold sumsets. We obtain numerous positive and negative results, but also leave open many interesting cases.

We also show that in most cases if we have a specific example of such a compact set then actually the generic (in the sense of Baire category) compact set in a suitably chosen space is also an example.

Finally, utilizing a computer-aided construction, we show that the compact set in the plane with full projections and nowhere dense sumset can be self-similar.
\end{abstract}

\maketitle

\tableofcontents

\section{Introduction}

Many intensively studied problems in additive combinatorics ask how large a set $K$ can be if its Minkowski sumset $K+K$ is small in some sense. Somewhat similarly, numerous classical problems and results in geometric measure theory relate the size (e.g.~Hausdorff measure or dimension) of a set $K$ to the size of its projections. Working on a seemingly unrelated problem, namely on the continuity of additive and convex functions, Banakh, Jab\l{}o\'nska and Jab\l{}o\'nski have arrived at a question that is an interesting combination of the two above mentioned themes. More precisely, they have asked the following, see \cite[Problem 2.7]{Banakh}.

\begin{question}[\textbf{Banakh--Jab\l{}o\'nska--Jab\l{}o\'nski}] 
Is there a compact set $K \subseteq \R^2$ such that the projection of $K$ on each line is of non-empty interior, but $K+K$ is nowhere dense?
\end{question}

Our first main result will be an answer to this question in arbitrary dimension. Recall that in $\R^d$ a hyperplane is a $(d-1)$-dimensional affine subspace.

%From now on, let $d$ always denote the dimension of the ambient space $\R^d$, and let us always assume $d \ge 2$. DONÁT KÉRTE EZT A MONDATOT, DE KELL EZ? NEM – TÖRÖLHETŐ

\begin{restate}{t:strong}
For every $d \ge 2$ there exists a compact set $K \subseteq [0,1]^d$ such that the projection of $K$ is of non-empty interior on every hyperplane, but $K+K$ is nowhere dense.
\end{restate}

The proof will be based on a random construction.

\bigskip
We will strengthen and generalize this result in various directions, which will naturally lead us to the following definition. For an affine subspace $V \subseteq \R^d$ let us denote the orthogonal projection on $V$ by $\pr_V$.

\begin{definition}
We say that a set $K \subseteq [0, 1]^d$ has \emph{full $m$-dimensional projections}, if $\pr_V (K) = \pr_V \left([0, 1]^d\right)$ for every $m$-dimensional affine subspace $V \subseteq \R^d$.
\end{definition}

The following fact is obvious.

\begin{fact}
\label{f:projections <-> intersections}
A set $K \subseteq [0, 1]^d$ has full $m$-dimensional projections iff $K$ intersects each $(d-m)$-dimensional affine subspace that meets $[0,1]^d$.
\end{fact}

Clearly, in \autoref{t:strong} hyperplanes can be replaced by subspaces of smaller dimension, as well. However, as our next main result will show, the analogous results are not true for full projections (see Sections \ref{sec:full projections} and \ref{sec:negative}). The threshold for the dimension of subspaces will turn out to be $\frac{d}2$.

\begin{restate}{c:(2,k,d) <=> k >= d/2}
The following are equivalent:
\begin{enumerate}[label=\normalfont{(\roman*)}]
\item $k \ge \frac{d}2$,
\item there exists a compact set $K \subseteq [0, 1]^d$ such that $K$ intersects each $k$-dimensional affine subspace that meets $[0,1]^d$, but $K + K$ is nowhere dense.
\end{enumerate}
\end{restate}
The picture for $\ell$-fold sumsets instead of $K+K$ is much less complete. For brevity, let us introduce the following notation.

\begin{definition}
For $\ell, k, d \in \N$, $\ell \ge 2$, $k < d$, $d \ge 2$ we say that $K \subseteq \R^d$ is an \emph{$(\ell, k, d)$-set}, if $K$ is compact, $K$ intersects each $k$-dimensional affine subspace that meets $[0,1]^d$, but the $\ell$-fold sumset $\ell K = K+\dots +K$ ($\ell$ times) is nowhere dense.
\end{definition}

Then, for example, the second main result above states that there exists a $(2, k, d)$-set iff $k \ge \frac{d}2$.

\bigskip
From now on, throughout the paper, let $\ell, k, d$ be as in the above definition.
\bigskip

For $\ell$-fold sumsets we obtain numerous positive and negative partial results in Sections \ref{sec:multiple} and \ref{sec:negative}, for example we show that $\exists (d, k, d)$-set iff $\exists (d-1, k, d)$-set iff $k = d-1$, but the general case is left open.

We also show in Section \ref{sec:generic} that if $\exists (\ell, k, d)$-set then actually the generic (in the sense of Baire category) compact set in $\{K \subseteq [0, 1]^d : K \textrm{ is compact and } $  $ \textrm{ intersects each } k\textrm{-dimensional affine subspace that meets } [0,1]^d \}$ is also an $(\ell, k, d)$-set.

Then, by an explicit construction that was found with the aid of a computer search, we show in Section \ref{sec:self-similar} that there exists a self-similar $(2, 1, 2)$-set. This self-contained proof only works in the planar case, but it avoids both the probabilistic method and the combinatorial lemmas from Section \ref{sec:line traces}.

Finally, we mention some open problems in Section \ref{sec:open}.

\section{Preliminaries}

For a set $K \subseteq \R^d$ and a natural number $\ell \ge 2$, we denote the $\ell$-fold sumset as
\[\ell K = \{k_1 + \dots + k_l : k_1, \dots, k_l \in K\}.\]
We should caution the reader that the iterated sumset $\ell K$ is in general not the same as the dilate $\lambda\cdot K = \{\lambda x : x \in K\}$, where $\lambda\in\R$ is not necessarily an integer.

We will construct the sets $K$ with the desired properties by approximating them in gradually refined grids of (hyper)cubes. (Note that although we work in an arbitrary dimension $d\ge 2$, for the sake of brevity we use the word `cube' instead of `hypercube'.) The following notions allow us to refer to the grid cubes and their vertices.

Given $d \in \N$, $n \in \N \setminus \{0\}$ and $t_1, t_2, \dots, t_d \in \Z$, let $C^d_n(t_1, \dots, t_d)$ denote the cube
\[C^d_n(t_1,\ldots,t_d)=\left[\frac{t_1}{n}, \frac{t_1 + 1}{n}\right] \times \dots \times \left[\frac{t_d}{n}, \frac{t_d + 1}{n}\right].\]

We will subdivide the unit cube $[0, 1]^d$ into a grid of these small cubes and introduce the notation
\[\mathcal{C}_n^d = \left\{ C^d_n(t_1, \dots, t_d) : t_i \in \Z \text{ and } 0 \le t_i < n \text{ for all $i=1, \dots, d$}\right\}\]
for the set of these small cubes.

We call a point $\left(\frac{t_1}{n}, \dots, \frac{t_d}{n}\right)$ a \emph{grid point} if $t_i \in \Z$ and $0 \le t_i \le n$ for each $i \le d$. (That is, the grid points are the vertices of the cubes in $\mathcal{C}_n^d$.)

For a point $x\in\R^d$ we write $x_i$ to denote the $i$th coordinate of $x$.

In various parts of the proof, we will use the notation `$\pr$' to denote several kinds of projections. As we already mentioned in the introduction, if $V\subseteq \R^d$ is an affine subspace, then $\pr_V \colon \R^d\to V$ is the orthogonal projection onto $V$. With a slight abuse of notation, if $1 \le i\le d$ is an index, then we use $\pr_i \colon \R^d\to \R$ to denote the projection map $\pr_i(x)=x_i$ onto the $i$th coordinate. Moreover, in the proof of \autoref{cor:number of line traces} we will also use projection maps $\pr_{\overline{i}} \colon \mathcal{C}^d_n\to\mathcal{C}^{d-1}_n$ (where $1\le i\le d$ is an integer) that can be defined as
\[\pr_{\overline{i}}(C^d_n(t_1, \dots, t_{i-1}, t_i, t_{i+1}, \ldots, t_d))=C^{d-1}_n(t_1, \dots, t_{i-1}, t_{i+1}, \ldots, t_d).\]

\section{Line traces}\label{sec:line traces}

In a cube grid a line $L$ is represented by its trace: the grid cubes intersected by $L$. (Note that $L$ is an arbitrary line, we do not assume that it passes through any grid points.) 

\begin{definition}\label{d:associated line trace}
For a line $L \subseteq \R^d$ that intersects $[0, 1]^d$ and $n\ge 1$, the \emph{associated line trace} is $\lt_n(L) = \{C \in \mathcal{C}_n^d : C \cap L \neq \emptyset\}$.
\end{definition}

This representation will be useful in the random constructions described in \autoref{sec:non-empty interior} and \autoref{sec:full projections}, because each of the infinitely many lines intersecting $[0,1]^d$ is represented by one of the finitely many possible traces.

\begin{definition}\label{d:line trace}
We call a non-empty subset $S$ of $\mathcal{C}_n^d$ a \emph{line trace} if there is a line $L \subseteq \R^d$ such that $\lt_n(L) = S$. We denote the number of line traces in $\mathcal{C}^d_n$ by $\LT^{(d)}(n)$. 
\end{definition}

In the rest of this section we show that $\LT^{(d)}(n)$ grows polynomially in $n$ (for a fixed dimension $d$). 

\begin{lemma}
  \label{l:number of line traces 2d}
In two dimensions $\LT^{(2)}(n)\le (2n)^5$.
\end{lemma}

\begin{proof}
  Fix $n \ge 1$. To count the number of line traces in $\mathcal{C}^2_n$, first look at those that can be obtained in the form of $\lt_n(L)$ where 
\begin{align}
\begin{split}
  \label{e:ell positive slope, not grid points}
  &\text{$L$ is a line that has a non-negative (possible infinite) slope,}\\ &\text{and it does not pass through any grid points.}
\end{split}
\end{align}
  The second condition implies that there is a natural enumeration of the squares in $\mathcal{C}^2_n$ that $L$ intersects. Let us enumerate these squares as $C_0, C_1, \dots, C_m$, where $C_0$ contains the bottom left endpoint of the segment $[0,1]^2\cap L$ and for each $i < m$, the cubes $C_i$ and $C_{i + 1}$ share a common side.
  
  Using that the slope of $L$ is non-negative, for each $i < m$, one has to go either up or to the right from $C_i$ to get $C_{i + 1}$. Let us record the sequence of moves as a sequence $s \in \{U, R\}^m$, where $U$ represents a move up, and $R$ represents a move to the right. It is well-known (see e.g.~\cite{Balanced words}) that $s$ is \emph{balanced}, that is, for each $m' \le m$, if we take two subsequences of $s$, both consisting of consecutive elements and both having length $m'$, then the number of $U$s in the two sequences are either equal, or differ by $1$. 
  
  Since the number of $U$s and the number of $R$s in the sequence $s$ are both at most $n - 1$, we obtain that $m \le 2n - 2$. So to get all line traces $\lt_n(L)$ where $L$ satisfies \eqref{e:ell positive slope, not grid points}, one has to pick the square where $L$ enters the unit square and then check all balanced sequences. Using that the slope is non-negative, $L$ has to enter the unit square on the bottom or on the left side. This gives us $2n - 1$ many options for the first grid square it intersects. The number of balanced words of length $2n - 2$ is $1 + \sum_{i = 1}^{2n - 2} (2n - 1 - i) \varphi(i)$, where $\varphi$ is Euler's totient function (this was established by Lipatov \cite{Lipatov}, see also \cite{balanced}). Then the number of line traces coming from a line satisfying \eqref{e:ell positive slope, not grid points} is at most 
\begin{align*}
  (2n - 1)\left(1 + \sum_{i = 1}^{2n - 2} (2n - 1 - i) \varphi(i) \right)\le (2n - 1)\left(1 + \sum_{i = 1}^{2n - 2} (2n - 2) i \right) \\ = (2n - 1)\left(1 + \frac{1}{2}(2n - 1)(2n - 2)^2\right) \le \frac{1}{2}(2n - 1)^4,
\end{align*}
  for $n \ge 2$. If we allow $L$ to have negative slope then the number of line traces we get will be at most twice the previous bound, hence at most $(2n - 1)^4$. Notice that this bound now works for $n = 1$ as well. 
  
  The only thing that remains to count every line trace of $\mathcal{C}^d_n$, is to count those $\lt_n(L)$ that come from a line $L$ passing through grid points. We only have finitely many lines $L$ that passes through at least two grid points inside the unit square, their number is at most $(n + 1)^4$, since the number of grid points in the unit square is $(n + 1)^2$. 
  
  Now suppose that $L$ passes through exactly one grid point $p$. It is easy to see that we can move $L$ to another line $L'$ such that $L'$ does not pass through any grid points, moreover, $L'$ intersects exactly those grid squares that $L$ intersects, except maybe for one grid square that contains $p$. In particular, $\lt_n(L') \subseteq \lt_n(L)$ and $|\lt_n(L')| \ge |\lt_n(L)| - 1$.

Hence, to get all line traces $\lt_n(L)$ where $L$ passes through exactly one grid point, it is enough to consider the line traces of the form $\lt_n(L')$ with $L'$ not passing through any grid points, and add a grid square that has a common side with one of the squares in $\lt_n(L')$. If the absolute value of the slope of $L'$ is at most $1$ then each adjacent square is above or below a square in $\lt_n(L')$, hence the number of adjacent squares is at most $2n$.

An analogous argument works if the absolute value of the slope is at least $1$, hence the number of line traces in $\mathcal{C}^2_n$ is at most 
\begin{align*}
    LT^{(2)}(n) &\le (2n + 1)(2n - 1)^4 + (n + 1)^4 = (2n)^5 - 47n^4 + 20n^3 + 14n^2 - 2n + 2 \\ &\le (2n)^5
\end{align*}
for $n \ge 1$, proving the lemma.
\end{proof}

Recall that for an index $1\le i \le d$, $\pr_{\overline{i}} \colon \mathcal{C}^d_n \to \mathcal{C}^{d-1}_n$ denotes the projection onto the hyperplane orthogonal to the $i$th coordinate axis.

\begin{claim}
\label{c:projections of traces}
  The projections $\{\pr_{\overline{i}}(\lt_n(L))\}_{1\le i\le d}$ determine $\lt_n(L)$, that is, if the lines $L, L' \subseteq \R^d$ satisfy $\pr_{\overline{i}}(\lt_n(L)) = \pr_{\overline{i}}(\lt_n(L'))$ for each $1 \le i \le d$, then $\lt_n(L) = \lt_n(L')$. 
\end{claim}
\begin{proof}
  Suppose we have two line traces $\lt_n(L)$ and $\lt_n(L')$ with $\pr_{\overline{i}}(\lt_n(L)) = \pr_{\overline{i}}(\lt_n(L'))$ for all $i \le d$, we need to show that $\lt_n(L) = \lt_n(L')$. By symmetry, it is enough to show that if $C = C^d_n(t_1, \dots, t_d) \in \lt_n(L')$ then $C \in \lt_n(L)$. Since the projections of $\lt_n(L)$ and $\lt_n(L')$ are the same, for each $i \le d$ we can find a cube $C^i \in \lt_n(L)$ with $\pr_{\overline{i}}(C^i) = \pr_{\overline{i}}(C)$, that is, $C^i = C^d_n(t_1, \dots, t_{i - 1}, b_i, t_{i+1}, \dots, t_d)$ for some $b_i$.
  
  The intersections of the form $C^i \cap L$ are compact convex subsets of $L$. After identifying $L$ with $\R$, we can view these intersections as closed bounded intervals, and denote the left endpoint of $C^i \cap L$ as $p^i$. We can also order the left endpoints in a non-decreasing order as $p^{i_1}, p^{i_2}, \dots, p^{i_d}$. It is enough to show that $p^{i_2} \in C$, since then $L \cap C \neq \emptyset$ and therefore $C \in \lt_n(L)$. 
  
  Notice that each $1 \le r \le d$ satisfies that $p^{i_r}_j \in [\frac{t_j}{n}, \frac{t_j + 1}{n}]$ for each $j \neq i_r$, because $p^{i_r} \in C^{i_r}$. In particular, $p^{i_2}_j \in [\frac{t_j}{n}, \frac{t_j + 1}{n}]$ for each $j \neq i_2$, and $p^{i_1}_{i_2}, p^{i_3}_{i_2} \in [\frac{t_{i_2}}{n}, \frac{t_{i_2} + 1}{n}]$. However, $p^{i_2}$ is a convex combination of the points $p^{i_1}$ and $p^{i_3}$, therefore $p^{i_2}_{i_2} \in [\frac{t_{i_2}}{n}, \frac{t_{i_2} + 1}{n}]$. It follows that $p^{i_2}$ is indeed in $C$, completing the proof.
\end{proof}

\begin{corollary}
  \label{cor:number of line traces}
  For $d \ge 2$ and $n \ge 1$, 
  $$\LT^{(d)}(n) \le (2n)^{\frac{5}{2}d!} .$$
\end{corollary}
\begin{proof}
  The two-dimensional case was handled in \autoref{l:number of line traces 2d}.  Now let $d > 2$, and assume that $L$ is a line in $\R^d$ that intersects the unit cube. 

  Using \autoref{c:projections of traces}, the number of line traces in $\mathcal{C}^d_n$ is at most 
  $$
    \LT^{(d)}(n) \le (\LT^{(d-1)}(n))^{d} \le (2n)^{5d\frac{(d-1)!}{2}} = (2n)^{5\frac{d!}{2}},
  $$
  proving the corollary by induction. 
\end{proof}

\section{Strong intersections}

The idea behind our construction is that if a line $L$ passes through a grid cube $C$ and we subdivide $C$ into smaller cubes, then $L$ `usually' intersects many of these smaller cubes and we may discard a large portion of the smaller cubes.

Unfortunately there are exceptions: for example if $|L\cap C|=1$ (e.g.~$L$ passes through a vertex of $C$), then $L$ cannot intersect more than one small cube from the subdivision. To rule out cases like this, we will only consider \emph{strong intersections} where the segment $L\cap C$ is large enough compared to the side length of the cube $C$. To simplify some calculations, we will quantify this `large enough' with the $\ell^\infty$ 
(i.e.~maximum) norm:

\begin{definition}\label{strong intersection def}
Assume that $L$ is a line and $C$ is a cube with axis-parallel edges in $\R^d$. We say that $L$ \emph{strongly intersects} $C$ when
\[\diaminfty(L\cap C) \ge \frac{\diaminfty(C)}{2d},\]
where $\diaminfty$ denotes the diameter in the $\ell^\infty$-norm.
\end{definition}

Recall that if $S\subseteq \R^d$ is a convex compact set (e.g.~a line segment) and $\pr_i$ denotes the projection onto the $i$th coordinate axis, then $\pr_i(S)$ is clearly a segment for each $1 \le i\le d$ and 
\[\diaminfty(S)=\max_{1\le i\le d} \length(\pr_i(S))\]
is just the length of the longest possible projection of $S$ onto a coordinate axis. In particular, if $C$ is a cube with axis-parallel edges, then $\diaminfty(C)$ is simply the side length of $C$.

It is easy to verify that \autoref{strong intersection def} implies the following fact:
\begin{fact}\label{strong intersection invariance}
Strong intersection is invariant under translations and scaling. In other words, if $\phi \colon \R^d\to \R^d$ is defined as $\phi(x) = \lambda x + t$ for some $\lambda>0$ and $t\in\R^d$, then $L$ strongly intersects $C$ if and only of $\phi(L)$ strongly intersects $\phi(C)$.
\end{fact}

Let us introduce the notation
\[\mathcal{L} = \{L: \text{ L is a line that strongly intersects $[0,1]^d$}\}.\]

Analogously to \autoref{d:associated line trace}, we can introduce the following notion:
\begin{definition}\label{d:associated strong line trace}
For a line $L\in \mathcal{L}$ that strongly intersects $[0,1]^d$, the \emph{associated strong line trace} is
\[\slt_n(L)=\{C \in \mathcal{C}^d_n : L\text{ strongly intersects }C\}.\]
\end{definition}

The main idea behind the notion of strong intersection is the following claim:
\begin{claim}\label{strong intersection inheritance naive}
If $L\in\mathcal{L}$, then $|\slt_n(L)|\ge \frac{n}{2d}-2$.
\end{claim}

For the random construction described in \autoref{sec:non-empty interior} we would also need to show that the number of traces is small enough, i.e.~polynomial in $n$ (for a fixed $d$). Instead of directly proving this for the map $\slt_n \colon \mathcal{L} \to \mathcal{P}(\mathcal{C}^d_n)$, we will prove \autoref{strong intersection inheritance} which is sufficient for our goals (and immediately implies \autoref{strong intersection inheritance naive}).

\begin{lemma}\label{strong intersection inheritance}
For each positive integer $n$, there exists a map $\slt'_n \colon \mathcal{L} \to \mathcal{P}(\mathcal{C}^d_n)$ with the following properties:
\begin{enumerate}[label=\normalfont{(\roman*)}]
\item if $L\in \mathcal{L}$, then $\slt'_n(L) \subseteq \slt_n(L)$ is a set of cubes strongly intersecting $L$,
\item for any $L\in\mathcal{L}$, $\slt'_n(L)$ is `large enough': there exists a coordinate $1\le r \le d$ such that 
\[|\{\pr_r(C): C\in \slt'_n(L)\}|\ge \frac{n}{2d}-2,\]% (in particular, $|\slt'_n(L)|\ge \frac{n}{2d}-2$), 
\item the range of the map $\slt'_n$ is small: 
\[|\{\slt'_n(L) : L\in \mathcal{L}\}| \le d (8dn)^{\frac{5}{2}d!}.\]
\end{enumerate}
\end{lemma}

\begin{proof}
For a segment $S$ in $\R^d$, let us define the associated \emph{reference coordinate} as
\[\refax(S)=\min \{i : 1\le i \le d, \diaminfty(S) = \length(\pr_i(S))\}.\]
It is clear that if $L$ is a line in $\R^d$ and $S_1, S_2\subseteq L$ are segments of positive length, then $\refax(S_1)=\refax(S_2)$. Using this, we can define the \emph{reference coordinate} of a line $L$ in $\R^d$ as
\[\refax(L)=\refax(S),\text{ where $S\subseteq L$ is an arbitrary segment of positive length}.\]

Fix an integer $n\ge 1$ and a line $L\in \mathcal{L}$ that strongly intersects $[0,1]^d$. We will define the set $\slt'_n(L)$ of grid cubes that will satisfy properties (i), (ii) and (iii).

In this proof we will work with two cube grids: our goal is to select grid cubes of the \emph{coarse} grid $\mathcal{C}^d_n$ that are strongly intersected by $L$, but we will also study the line trace of $L$ in the \emph{fine} grid $\mathcal{C}^d_{4dn}$. Each cube $C\in \mathcal{C}^d_n$ from the coarse grid can be written as the union of $(4d)^d$ cubes from the fine grid; we will denote the set of these cubes by
\[[C]=\{C' \in \mathcal{C}^d_{4dn} : C'\subseteq C\}.\]

Now we are ready to define $\slt'_n(L)$ as the set of cubes $C\in \mathcal{C}^d_n$ that satisfy
\begin{equation} \label{eq:sltprime def}
\{ \pr_{\refax (L)}(C^*): C^*\in [C]\cap \lt_{4dn}(L)\}|\ge 4.
\end{equation}

This formula only depends on $L$ through the index $1 \le \refax(L)\le d$ and the line trace $\lt_{4dn}(L)$, therefore the map $\slt'_n$ may take at most $d\LT^{(d)}(4dn)$ many different values. This implies that property (iii) holds, because according to \autoref{cor:number of line traces}, $\LT^{(d)}(4dn)\le (8dn)^{\frac52 d!}$.

After this, we prove that property (i) is satisfied, that is, \eqref{eq:sltprime def} implies that $L$ strongly intersects $C$.

For a fine grid cube $C^* = C^d_{4dn}(t_1, \ldots, t_d)\in \mathcal{C}^d_{4dn}$ it is clear from the definitions that $\pr_{\refax(L)}(C^*) = \left[\frac{t_{\refax(L)}}{4dn}, \frac{t_{\refax(L)}+1}{4dn}\right]$ is determined by the integer $t_{\refax(L)}$. Therefore, \eqref{eq:sltprime def} implies that there are (fine) grid cubes
\[C^d_{4d n}(s_1, \ldots, s_d), C^d_{4dn}(t_1, \ldots, t_d)\in [C]\cap \lt_{4dn}(L)\]
such that $s_{\refax(L)}-t_{\refax(L)}\ge 3$. Pick points
\[p\in C^d_{4dn}(s_1, \ldots, s_d)\cap L\qquad\text{and}\qquad q\in C^d_{4dn}(t_1, \ldots, t_d)\cap L.\]
Clearly $p, q\in L\cap C$ and we can estimate the relevant coordinates of $p$ and $q$ as
\[p_{\refax(L)}\ge \frac{s_{\refax(L)}}{4dn}\quad\text{and}\quad\frac{t_{\refax(L)}+1}{4d n} \ge q_{\refax(L)}.\]
These show that
\[\diaminfty(L\cap C)=\length(\pr_{\refax(L)}(L\cap C))\ge p_{\refax(L)}-q_{\refax(L)} \ge \frac{2}{4dn} =\frac{\diaminfty(C)}{2d}\]
and therefore $L$ strongly intersects $C$.

It remains to show that property (ii) is satisfied. As $L$ strongly intersects $[0,1]^d$, we know that
\[\length(\pr_{\refax(L)}(L\cap[0,1]^d)) = \diaminfty(L\cap[0,1]^d)\ge \frac{\diaminfty([0,1]^d)}{2d}=\frac{1}{2d}.\]
Using this, it is easy to verify that $|B|\ge \frac{n}{2d}-2$ for the set
\[B=\left\{b\in\N : 0\le b <n, \left[\frac{b}{n}, \frac{b+1}{n}\right]\subseteq \pr_{\refax(L)}(L\cap[0,1]^d)\right\}.\]

We will prove that for each $b\in B$ there exists at least one grid cube $C\in\slt'_n(L)$ such that $\pr_{\refax(L)}(C)=[\frac{b}{n}, \frac{b+1}{n}]$. As $|B|\ge \frac{n}{2d}-2$, this immediately implies property (ii).

Fix $b\in B$ and consider the open segment
\[S_b = \left\{x\in L : \frac{b}{n}< x_{\refax(L)} < \frac{b+1}{n}\right\}.\]
This segment is contained in the unit cube $[0,1]^d$, because we know that $\left[\frac{b}{n}, \frac{b+1}{n}\right]\subseteq \pr_{\refax(L)}(L\cap[0,1]^d)$ (and the restriction $\pr_{\refax(L)}\upharpoonright_L$ is a bijection).

For $1\le j\le d$ consider the set
\[P_{j}=\left\{x\in\R^d: x_j=\frac{a}{n}\text{ for some } a\in\mathbb{N}, 0\le a \le n\right\}\]
(which is the union of $n+1$ parallel `cutting' hyperplanes). These sets describe the coarse grid $\mathcal{C}^d_n$ in the sense that the bounded connected components of $\R^d\setminus \bigcup_{j=1}^d P_j$ are the interiors of cubes $C\in \mathcal{C}^d_n$.

We say that a point $p\in S_b$ is a \emph{partitioning point} if there is $1\le j\le d$ such that $p\in P_j$, but $L\nsubseteq P_j$. It is clear that for any $1\le j\le d$,
\[\length(\pr_j(S_b))\le \diaminfty(S_b)=\length(\pr_{\refax(L)}(S_b))= \frac1n\]
and therefore if $L \nsubseteq P_j$, then $|S_b \cap P_j|\le 1$ (also using the fact that $S_b$ is an open segment). Moreover, it is clear from the definition of $S_b$ that $S_b\cap P_{\refax(L)}=\emptyset$. These show that there are at most $d-1$ partitioning points in the open segment $S_b$. 

This implies that we can find an open segment $S\subseteq S_b$ which does not contain any partitioning points and satisfies that
\[\diaminfty(S)\ge \frac{\diaminfty(S_b)}{d}=\frac{1}{dn}.\]
It is straightforward to verify that this segment $S$ must be contained in the closure $\overline{U}$ of a (not necessarily uniquely determined) connected component $U$ of $\R^d\setminus \bigcup_{j=1}^d P_j$.

As $S\subseteq S_b\subseteq[0,1]^d$, we may assume that this component $U$ is a bounded one, which means that the cube $C= \overline{U} \in \mathcal{C}^d_n$ satisfies that $S \subseteq C$. The definition of $S_b$ clearly implies that $\pr_{\refax(L)}(C)=[\frac{b}{n}, \frac{b+1}{n}]$.

Notice that the segment $S\subseteq L \cap C$ is covered by the elements of $[C]\cap \lt_{4dn}(L)$ (which are cubes from the fine grid) and therefore
\[\pr_{\refax(L)}(S) \subseteq \bigcup_{C^*\in [C]\cap\lt_{4dn}(L)}\pr_{\refax(L)}(C^*).\]
As the cubes $C^*\in [C]\cap\lt_{4dn}(L)$ are from the fine grid, their projections under $\pr_{\refax(L)}$ are all intervals of length $\frac{1}{4dn}$. These projections cover the interval $\pr_{\refax(L)}(S)$, but the choice of $S$ guarantees that
\[\length(\pr_{\refax(L)}(S))= \diaminfty(S) \ge \frac{1}{dn},\]
therefore this can only happen if there are at least four distinct intervals of the form $\pr_{\refax(L)}(C^*)$ for $C^*\in [C]\cap\lt_{4dn}(L)$, i.e.
\[|\{ \pr_{\refax (L)}(C^*): C^*\in [C]\cap \lt_{4dn}(L)\}|\ge 4.\]
But this is the condition \eqref{eq:sltprime def}
for the cube $C$, so we proved that $C\in\slt'_n(L)$. As this grid cube $C$ corresponds to the arbitrary element $b\in B$, we have proved that property (ii) is satisfied.
\end{proof}

We will also need the following straightforward lemma:
\begin{lemma}\label{l:non-empty interior}
Consider a cube $C$ with axis-parallel edges and assume that $K\subseteq \R^d$ satisfies that if a line $L$ strongly intersects $C$, then $L$ intersects $K$. Then for each hyperplane $H$, the projection $\pr_H(K)$ has non-empty interior on $H$.
\end{lemma}

\begin{proof}
Let $o$ denote the center of $C$ and consider a homothety with fixed point $o$ and similarity ratio $\left(1-\frac{1}{d}\right)$. This transformation shrinks $C$ into the cube
\[C' = o + \left(1-\frac{1}{d}\right)\cdot(C-o).\]

Assume that a line $L$ intersects $C'$. As $C'$ is contained in $C$, it is clear that $L$ also intersects the boundary $\partial C$ of the cube $C$. This implies that
\[\diaminfty(L\cap C)\ge \frac{1}{2d} \diaminfty(C),\]
because it is easy to verify that the $\ell^\infty$-distance between $C'$ and $\partial C$ is $\frac{1}{2d}\diaminfty(C)$.

This shows that for an arbitrary line $L$, if $L$ intersects $C'$, then $L$ strongly intersects $C$, and therefore $L$ intersects $K$ according to our assumption. Applying this observation to the lines orthogonal to $H$ yields that $\pr_H(C')\subseteq \pr_H(K)$. As $C'$ is a cube, its projection has non-empty interior, concluding our proof.
\end{proof}

\section{Projections with non-empty interior}\label{sec:non-empty interior}

We will prove the following theorem:
\begin{theorem}
\label{t:strong}
For every $d \ge 2$ there exists a compact set $K \subseteq [0,1]^d$ such that the projection of $K$ is of non-empty interior on every hyperplane, but $K+K$ is nowhere dense.
\end{theorem}

\begin{proof}
Our compact set $K$ is constructed as the intersection of sets $K_n$ such that each $K_n$ is a union of finitely many closed cubes from $\mathcal{C}^d_{a_n}$ for some $a_n$. 

Let $U_1, U_2, \dots$ be an enumeration of a countable basis of $[0, 1]^d$ which does not contain the empty set.

For each $m\ge1$, fix a map $\slt'_m$ with the properties guaranteed in \autoref{strong intersection inheritance}. (That is, the range of $\slt'_m$ has cardinality $\le d (8dm)^{\frac{5}{2}d!}$ and if $L$ strongly intersects $[0,1]^d$, then $\slt'_m(L)\subseteq\mathcal{C}^d_m$ contains cubes which are all strongly intersected by $L$ and there exists a coordinate $1\le r\le d$ such that $|\{\pr_r(C) : C\in \slt'_m(L)\}|\ge\frac{m}{2d}-2$.)

%In the trivial case of $m=1$ let us define $\slt'_1(L)=\{[0,1]^d\}$ for any line $L$ that strongly intersects $[0,1]^d$. (The conditions of \autoref{strong intersection inheritance} would also allow $\slt'_1(L)=\emptyset$ for some lines, but that would cause technical difficulties.)

The constructed sequence of sets $K_0, K_1, \dots$ and sequence of positive integers $a_0, a_1, \dots$ will have the following properties for each $n \ge 0$:

\begin{enumerate}[label=(\roman*)]
\item there is a collection $\mathcal{C}_n \subseteq \mathcal{C}^d_{a_n}$ such that $K_n = \bigcup \mathcal{C}_n$, 
\item if $n\neq 0$, then $K_{n} \subseteq K_{n-1}$ and $a_n$ is divisible by $a_{n-1}$,
\item if a line $L$ strongly intersects $[0,1]^d$, it strongly intersects a cube $C\in\mathcal{C}_n$,
\item if $n\neq 0$, then $\frac{K_n + K_n}{2} \nsupseteq U_n$.
\end{enumerate}

Before describing this construction, we show that these properties imply that $K = \bigcap_n K_n$ satisfies the conditions of the theorem. If a line $L$ strongly intersects $[0,1]^d$, then (iii) implies that $K_n \cap  L \neq \emptyset$ for each $n\in\mathbb{N}$, as this is a decreasing sequence of non-empty compact sets, this implies that $K\cap L\neq\emptyset$. Now we can apply \autoref{l:non-empty interior} with $C = [0,1]^d$ and $K$ to see that the projection of $K$ is of non-empty interior on every hyperplane. On the other hand, (iv) clearly implies that $K+K$ is nowhere dense.

To start the construction, let $a_0=1$ and $K_0 = [0, 1]^d$. It is easy to check that conditions (i)--(iv) are all satisfied for $n = 0$. 

Now suppose that positive integers $a_0 < a_1 < \dots < a_n$ and compact sets $K_0 \supseteq K_1 \supseteq \dots \supseteq K_n$ are defined and satisfy conditions (i)--(iv). Our task is to define $a_{n + 1}$ and $K_{n + 1}$. To construct $K_{n + 1}$, we use a probabilistic argument that depends on the choice of $a_{n + 1}$. We first describe the construction, then show that if $a_{n + 1}$ is large enough, then $a_{n + 1}$ and the resulting set $K_{n + 1}$ satisfy (i)--(iv) with positive probability.

So let $a_{n + 1}$ be some fixed positive integer which is an integer multiple of $a_n$. According to (i), we need to define $K_{n+1}$ by selecting some grid cubes from the `fine' grid $\mathcal{C}^d_{a_{n+1}}$ and taking their union. To satisfy (ii), we may only select grid cubes from the set
\[\mathcal{T}_{n+1}=\{C \in \mathcal{C}^d_{a_{n+1}} : C\subseteq K_n\}.\]
As $U_{n+1}$ is open and the points with irrational coordinates form a dense set, we can fix a point $x\in U_{n+1} \cap (\R\setminus\mathbb{Q})^d$. During the construction we will ensure that $x\notin \frac{K_n+K_n}{2}$ and this will clearly imply (iv).

Let us color each cube in $\mathcal{T}_{n+1}$ black with probability $\frac12$ independently of each other, and denote by $\mathcal{T}^B_{n+1}$ the set of black cubes.

We can define $\mathcal{C}_{n+1}$ as
\[\mathcal{C}_{n+1}=\left\{C\in \mathcal{T}^B_{n+1}: \text{every } C'\in \mathcal{T}^B_{n+1}\text{ satisfies that }x\notin \frac{C+C'}{2}\right\}.\]
It is clear that for any $a_{n+1}$, which is divisible by $a_n$, the number $a_{n+1}$ and the set $K_{n+1}=\bigcup \mathcal{C}_{n+1}$ satisfy properties (i), (ii) and (iv) for any possible coloring of the cubes. To conclude the proof, we show that if $a_{n+1}$ is large enough, then $K_{n+1}$ satisfies (iii) with positive probability.

For a cube $C\in \mathcal{T}_{n+1}$, the event $\{C\in \mathcal{C}_{n+1}\}$ is determined by the coloring of the cube $C$ and the cubes in 
\[\mathcal{E}_C = \left\{C'\in \mathcal{T}_{n+1}: x\in \frac{C+C'}{2}\right\},\]
therefore this event is independent of the coloring of cubes $\tilde{C}\notin \{C\}\cup \mathcal{E}_C$. Notice that $|\mathcal{E}_C|\le 2^d$, because the condition $x\in \frac{C + C'}{2}$ is equivalent to $C'\cap (2x-C) \neq \emptyset$ and we assumed that $x$ has no rational coordinates. This shows that
\[\mathbb{P}(C\in \mathcal{C}_{n+1})\ge \frac{1}{2^{2^d +1}}.\]

As property (iii) holds for $n$, we know that if a line $L$ strongly intersects $[0,1]^d$, then it strongly intersects a cube $C_L\in \mathcal{C}_n$. For each line $L$ that strongly intersects $[0,1]^d$, let us fix one such $C_L$. Let $\psi_L$ denote the homothety which maps $C_L$ onto $[0,1]^d$ and has a positive similarity ratio $a_n$.

As $a_n$ is a divisor of $a_{n+1}$, we may note that the homothety $\psi_L$ induces a bijection between the grids
\[\{C'\in \mathcal{C}^d_{a_{n+1}} : C'\subseteq C_L\}\text{ (subdividing $C_L$)}\quad\text{and}\quad \mathcal{C}^d_{\frac{a_{n+1}}{a_n}} \text{ (subdividing $[0,1]^d$)}.\]

According to \autoref{strong intersection invariance}, $\psi_L(L)$ strongly intersects $[0,1]^d$ and therefore we may consider the family
\[\mathcal{F}_L = \left\{C\in \mathcal{C}^d_{a_{n+1}} : C\subseteq C_L \text{ and } \psi_L(C)\in \slt'_{\frac{a_{n+1}}{a_n}}(\psi_L(L))\right\}\]
of grid cubes. Note that $C_L\in \mathcal{C}_n$ means that $C_L\subseteq K_n$, and this implies that $\mathcal{F}_L\subseteq \mathcal{T}_{n+1} = \{C \in \mathcal{C}^d_{a_{n+1}} : C\subseteq K_n\}$.

We know that the family $\mathcal{F}_L$ satisfies that
\[\exists r\in\{1,2, \ldots, d\}: \quad |\{\pr_r(C): C\in \mathcal{F}_L\}|\ge \frac{\frac{a_{n+1}}{a_n}}{2d} -2= \frac{a_{n+1}}{2 d a_n} -2\] 
because \autoref{strong intersection inheritance} guarantees that the maps $\slt'_m$ ($m = 1, 2, \ldots$) satisfy an analogous condition and $\psi_L$ preserves the relation that two cubes have the same projection onto a certain coordinate.

We want to define a subfamily $\mathcal{F}''_L\subseteq \mathcal{F}_L$ which is `large' but satisfies that the events $(\{C\in \mathcal{C}_{n+1}\})_{C \in \mathcal{F}''_L}$ are independent.

First, notice that the projections $\pr_r(C)$ (where $C\in \mathcal{F}_L$) are all closed intervals of the form $\left[\frac{i}{a_{n+1}}, \frac{i+1}{a_{n+1}}\right]$ for some integers $i\in\Z$. Therefore we may select a set $\mathcal{F}'_L$ that contains at least $\frac{1}{2}\left(\frac{a_{n+1}}{2 da_n} -2\right)$ many cubes and satisfies that $\pr_r(C_1)\cap \pr_r(C_2)=\emptyset$ for any two distinct cubes $C_1, C_2\in \mathcal{F}'_L$. It is straightforward to verify that this implies that $\mathcal{E}_{C_1}$ is disjoint from $\mathcal{E}_{C_2}$.

Among the selected cubes there may be at most one cube $C^* \in \mathcal{F}'_L$ which satisfies that $\pr_r(x)\in\pr_r(C^*)$. With the possible exception of this $C^*$, partition the selected cubes into the sets
\[\mathcal{F}^+_L = \{C\in \mathcal{F}'_L : \min (\pr_r(C))> x_r\}\]
and
\[\mathcal{F}^-_L = \{C\in \mathcal{F}'_L : \max (\pr_r(C))< x_r\}.\]
This partition guarantees that if $C_1, C_2\in \mathcal{F}^\varepsilon_L$ for  a sign $\varepsilon\in\{+,-\}$, then $x\notin\frac{C_1 + C_2}{2}$ and therefore $C_1\notin \mathcal{E}_{C_2}$ (and analogously $C_2\notin \mathcal{E}_{C_1}$).

If we pick the larger of these two sets as $\mathcal{F}''_L$, then we have
\[|\mathcal{F}''_L| \ge \frac{1}{2} \left(\frac{1}{2}\left(\frac{a_{n+1}}{2 d a_n} -2\right)-1\right)=\frac{a_{n+1}}{8d a_n} -1.\]
Moreover, the events $(\{C\in \mathcal{C}_{n+1}\})_{C \in \mathcal{F}''_L}$ are independent, because the system
\[\big\{\{C\}\cup\mathcal{E}_C : C\in \mathcal{F}''_L\big\}\]
consists of pairwise disjoint sets.

This independence implies that
\[\mathbb{P}(\mathcal{F}_L \cap \mathcal{C}_{n+1}= \emptyset) \le\mathbb{P}(\mathcal{F}''_L \cap \mathcal{C}_{n+1}= \emptyset) \le \left(1 -\frac{1}{2^{2^d +1}}\right)^{\frac{a_{n+1}}{8d a_n} -1}.\]

Although there are infinitely many lines that strongly intersect $[0,1]^d$, we may notice that there are at most $a_n^d$ choices for the cube $C_L\in \mathcal{C}_n\subseteq \mathcal{C}^d_{a_{n}}$. In each of these cubes, there are at most $d\!\left(8d \frac{a_{n+1}}{a_n}\right)^{\frac{5}{2}d!}$ possible choices for the set $\mathcal{S}_L$, because \autoref{strong intersection inheritance} guarantees that the maps $\slt'_m$ (for $m=1,2\ldots$) have small range.

This means that the probability of that (iii) is not satisfied (i.e.~there exists a line $L$ such that it strongly intersects $[0,1]^d$, but there is no $C\in \mathcal{C}_{n+1}$ which is strongly intersected by $L$) is at most
\[a_n^d d \left(8d \frac{a_{n+1}}{a_n}\right)^{\frac{5}{2}d!} \left(1 -\frac{1}{2^{2^d +1}}\right)^{\frac{a_{n+1}}{8da_n} -1}.\]
It is easy to see that if $a_{n+1}$ is large enough, then this probability is less than 1, which concludes our proof.
\end{proof}

\section{Constructions with full projections}\label{sec:full projections}

We now turn our attention to constructing compact sets $K \subseteq \R^d$ with $K+K$ nowhere dense such that $K$ has full $(d-k)$-dimensional projections, instead of only requiring that the projections have non-empty interior. However, as we will see in Section \ref{sec:negative}, one cannot construct such examples for each $k$ and $d$.

\bigskip

\begin{theorem} \label{t:K+K construction}
  If $k \ge \frac{d}{2}$, then there 
  exists a $(2, k, d)$-set, that is, a compact set $K \subseteq [0, 1]^d$ such that $K$ intersects each $k$-dimensional affine subspace that meets $[0,1]^d$, but $K + K$ is nowhere dense.
\end{theorem}

\begin{proof}
  Analogously to the proof of \autoref{t:strong}, our compact set $K$ is constructed as the intersection of sets $K_n$ such that each $K_n$ is a union of finitely many closed cubes from $\mathcal{C}^d_{a_n}$ for some $a_n$. 
  
  Let $U_1, U_2, \dots$ be an enumeration of a countable basis of $[0, 1]^d$ containing non-empty sets.
  The constructed sequence of sets $K_0, K_1, \dots$ and sequence of positive integers $a_0, a_1, \dots$ will have the following properties for each $n \ge 0$:
  
  \begin{enumerate}[label=(\roman*)]
  \item $K_n = \bigcup \mathcal{C}_n$ for some $\mathcal{C}_n \subseteq \mathcal{C}^d_{a_n}$, 
  \item if $n\neq 0$, then $K_{n} \subseteq K_{n-1}$ and $a_n$ is divisible by $20a_{n-1}$,
  \item $K_n \cap  V \neq \emptyset$ for each $k$-dimensional affine subspace $V$ that meets $[0, 1]^d$,
  \item if $n\neq 0$, then $\frac{K_n + K_n}{2} \not \supseteq U_n$,
  \item $U_{n + 1}$ contains a cube from $\mathcal{C}^d_{a_n}$. 
  \end{enumerate}
  
  Clearly, if we can carry out the construction in a way that the resulting sequence of sets satisfies the above properties, then $K = \bigcap_n K_n$ satisfies the condition of the theorem: For an affine subspace $V$ with $V \cap [0, 1]^d \neq \emptyset$, the sequence $(K_n \cap V)_{n}$ is decreasing and consists of non-empty compact sets, and therefore $K \cap  V \neq \emptyset$, while (iv) implies that $K + K$ is nowhere dense.
  
  To start the construction, let $a_0$ be large enough so that $U_1$ contains a cube from $\mathcal{C}^d_{a_0}$, and let $K_0 = \bigcup \mathcal{C}^d_{a_0} = [0, 1]^d$. It is easy to check that conditions (i)--(v) are all satisfied for $n = 0$. 
  
  Now suppose that positive integers $a_0 < a_1 < \dots < a_n$ and compact sets $K_0 \supseteq K_1 \supseteq \dots \supseteq K_n$ are defined and satisfy conditions (i)--(v). Our task is to define $a_{n + 1}$ and $K_{n + 1}$. We use a probabilistic argument to construct $K_{n + 1}$ that depends on $a_{n + 1}$. We first describe the construction, then show that if $a_{n + 1}$ is large enough, then $a_{n + 1}$ and the resulting set $K_{n + 1}$ satisfy (i)--(v) with positive probability.
  
  So let $a_{n + 1}$ be some fixed positive integer, a multiple of $20a_n$. Let
  \begin{align*}
    \textstyle\mathcal{S}_{n+1} = \{ C^d_{a_{n+1}}(t_1, \dots, t_d) \subseteq K_n: \textstyle|\{i \le d : \{ \frac{t_ia_n}{a_{n+1}}\} \in [0, \frac{1}{20}) \cup [\frac{19}{20}, 1)\}| \ge k + 1 \},
  \end{align*}
  \begin{align*}
    \textstyle
    \mathcal{T}_{n+1} = \{ C^d_{a_{n+1}}(t_1, \dots, t_d)  \subseteq (K_n \setminus \bigcup \mathcal{S}_{n+1}) : |\{i \le d : \{\frac{t_ia_n}{a_{n+1}}\} \in [\frac{9}{20}, \frac{13}{20})\}| \le k \},
  \end{align*}
  where $\{x\}$ denotes the fractional part of the number $x$. We note that since $a_{n+1}$ is a multiple of $20a_n$, for each cube $C = C^d_{a_{n+1}}(t_1, \dots, t_d)$, index $i \le d$, and integers $a < b$, 
  \begin{align}
      \label{e:conversion of cube and point coordinates}
      \textstyle
      \left\{\frac{t_ia_n}{a_{n+1}}\right\} \in \left[\frac{a}{20}, \frac{b}{20}\right) \qquad\Longleftrightarrow \qquad\forall x \in C \colon \{x_ia_n\} \in \left[\frac{a}{20}, \frac{b}{20}\right].
  \end{align}

  We select a cube $C_{n+1} \in \mathcal{C}^d_{a_{n+1}}$ with $C_{n+1} \subseteq U_{n+1}$ by choosing first a cube $C'_{n+1} = C^d_{a_n}(u_1, \dots, u_d) \in \mathcal{C}^d_{a_n}$ with $C'_{n+1} \subseteq U_{n+1}$ using (v). Then let $$\textstyle C_{n+1} = C^d_{a_{n+1}}\Big(\big(u_1 + \frac{1}{4}\big)\frac{a_{n+1}}{a_n}, \dots, \big(u_d + \frac{1}{4}\big)\frac{a_{n+1}}{a_n}\Big) \in \mathcal{C}^d_{a_{n+1}}.$$
  Since $a_{n+1}$ is a multiple of $20a_n$, the ``coordinates'' of $C_{n+1}$ are integers. One can also easily check that $C_{n+1} \subseteq C'_{n+1} \subseteq U_{n+1}$.   
  
  We now pause a bit in the proof and say a few words about the role of $\mathcal{S}_{n+1}$, $\mathcal{T}_{n+1}$ and $C_{n+1}$. The cube $C_{n+1}$ makes sure of satisfying (iv): the set $K_{n+1}$ will satisfy $\frac{K_{n+1} + K_{n+1}}{2} \cap \interior C_{n+1} = \emptyset$. $\mathcal{S}_{n+1}$ is some kind of neighborhood of the \mbox{$(d-k-1)$-ske}\-le\-ton of the cubes in $\mathcal{C}_n \subseteq \mathcal{C}_{a_n}^d$. The family of cubes $\mathcal{T}_{n+1}$ is chosen so that a $k$-dimensional affine subspace that meets $[0, 1]^d$ intersects either a cube from $\mathcal{S}_{n+1}$ or many cubes from $\mathcal{T}_{n+1}$; we will derive this fact from \autoref{l:subspace intersects skeleton or center}. The sets $\mathcal{S}_{n+1}$ and $\mathcal{T}_{n+1}$ also have the property
  \begin{align}
    \label{e:(S cap T)/2 disjoint from C}
    \frac{C^S + C^T}{2} \cap \interior C_{n+1} = \emptyset \text{ for all $C^S \in \mathcal{S}_{n+1}$ and $C^T \in \mathcal{T}_{n+1}$}.
  \end{align}
  To see this, note that \eqref{e:conversion of cube and point coordinates} implies $|\{ i\le d : \{x_ia_n\} \in [0, \frac{1}{20}] \cup [\frac{19}{20}, 1]\}|\ge k + 1$ for each $x \in \bigcup \mathcal{S}_{n+1}$,  and also that $|\{ i\le d : \{y_ia_n\} \in [\frac{9}{20},\frac{13}{20}]\}|\le k$ for each $y \in \bigcup \mathcal{T}_{n+1}$. It follows that there is an $i \le d$ with $\{x_ia_n\} \in [0, \frac{1}{20}] \cup [\frac{19}{20}, 1]$ and $\{y_ia_n\} \not \in [\frac{9}{20},\frac{13}{20}]$, therefore $$\textstyle \{\frac{x_i+y_i}{2}a_n\} \in [0, \frac{10}{40}] \cup [\frac{12}{40}, \frac{30}{40}] \cup [\frac{32}{40}, 1].$$ 
  Since $\{c_ia_n\} \in \left(\frac14, \frac14 + \frac1{20}\right)=\left(\frac{10}{40}, \frac{12}{40}\right)$ for each point $c \in \interior C_{n+1}$, \eqref{e:(S cap T)/2 disjoint from C} easily follows. 
  
  We now describe our probabilistic argument to leave out some cubes from $\mathcal{T}_{n+1}$ in order to satisfy (iv). We color each cube in $\mathcal{T}_{n+1}$ black with probability $\frac{1}{2}$ independently of each other, and denote by $\mathcal{T}_{n+1}^B$ the set of black cubes. Finally, let
  
  \begin{align*}\textstyle
    \mathcal{C}_{n+1} = \mathcal{S}_{n+1} \cup \left\{C \in \mathcal{T}_{n+1}^B : \forall C' \in \mathcal{T}_{n+1}^B \left(\frac{C+C'}{2} \cap \interior C_{n+1} = \emptyset\right) \right\}.    
  \end{align*}

  Now we show that if $a_{n+1}$ is sufficiently large then $K_{n+1} = \bigcup \mathcal{C}_{n+1}$ and $a_{n+1}$ satisfy the conditions (i)--(v) above for $n+1$ with positive probability. Clearly (i) and (ii) are satisfied. If $a_{n+1}$ is chosen to be large enough then $U_{n+2}$ contains a cube from $\mathcal{C}^d_{a_{n+1}}$, taking care of (v). 
  
  To check (iv), we show that if $x, x' \in K_{n+1}$ then $\frac{x+x'}{2} \not\in \interior C_{n+1} \subseteq U_{n+1}$. Let $C$ and $C'$ be cubes with $x\in C \in \mathcal{C}_{n+1}$ and $x' \in C' \in \mathcal{C}_{n+1}$. If $C, C' \in \mathcal{S}_{n+1}$ then using that $k \ge \frac{d}{2}$, there is an $i \le d$ with $\{x_ia_n\}, \{x'_ia_n\} \in [0, \frac{1}{20}] \cup [\frac{19}{20}, 1]$. It is then straightforward to check that $\frac{x+x'}{2}$ cannot coincide with a point $c \in \interior C_{n+1}$, since $c_i \in [\frac{5}{20}, \frac{6}{20}]$ for each $i$. 
  If $C, C' \in \mathcal{T}_{n+1}$ then $\frac{x+x'}{2} \not\in \interior C_{n+1}$ follows from our definition of $\mathcal{C}_{n+1}$, and if $C \in \mathcal{S}_{n+1}$ and $C' \in \mathcal{T}_{n+1}$ then it follows from \eqref{e:(S cap T)/2 disjoint from C}. 
  
  Therefore it remains to check (iii). Let $V$ be a $k$-dimensional affine subspace in $\R^d$ that intersects $[0, 1]^d$. Using (iii) for $n$, $V$ intersects some cube $C_V \in \mathcal{C}_n\subseteq \mathcal{C}^d_{a_n}$. We now use the following lemma to show that $V$ intersects either a cube from $\mathcal{S}_{n+1}$ or many cubes from $\mathcal{T}_{n+1}$.
  
  \begin{lemma}
    \label{l:subspace intersects skeleton or center}
  Suppose that $V$ is a $k$-dimensional affine subspace of $\R^d$ for some $k \ge \frac{d}{2}$ that meets $[0, 1]^d$. Then $V$ intersects either the set 
  $$
  \textstyle S = \left\{x \in [0, 1]^d : \left|\left\{i \le d : x_i \in \left[0, \frac{1}{20}\right] \cup \left[\frac{19}{20}, 1\right]\right\}\right| \ge k + 1\right\}, 
  $$
  or the set 
  $$
    \textstyle T = \left\{x \in \left[\frac{1}{20}, \frac{19}{20}\right]^d : \left|\left\{i \le d : x_i \in \left[\frac{8}{20}, \frac{14}{20}\right]\right\}\right| \le k \right\}.
  $$
  \end{lemma}
\theoremstyle{plain}

  We now continue with the proof of the theorem, and prove \autoref{l:subspace intersects skeleton or center} later. We can enlarge $C_V$ and $V$ using a homothety $\psi \colon \R^d \to \R^d$ with a ratio $a_n$ so that $\psi(C_V) = [0, 1]^d$, and then apply \autoref{l:subspace intersects skeleton or center}. Notice that $\psi(V)$ intersects the set $S$ of the lemma if and only if $V$ intersects a cube from $\mathcal{S}_{n+1}$.
  
  If $V$ does not intersect $\bigcup \mathcal{S}_{n+1}$, then $\psi(V)$ intersects the set $T$ of the lemma, meaning that there is a point $x \in V$ such that $\{x_ia_n\} \in [\frac{1}{20}, \frac{19}{20}]$ for all $i \le d$, and 
  $$\textstyle \left| \left\{ i \le d : \{x_ia_n\} \in \left[\frac{8}{20}, \frac{14}{20}\right]\right\}\right| \le k.$$ Let $S_V \subseteq V$ be any open segment of length $\frac{2}{20a_n}$ with midpoint $x$. Clearly, $S_V \subseteq C_V \subseteq K_n$, and $|\{i \le d : \{y_ia_n\} \in [\frac{9}{20}, \frac{13}{20}]\}| \le k$ for each $y \in S_V$. Now we show that 
  \begin{equation}
    \label{e:S_V subset T}
    S_V \subseteq \bigcup \mathcal{T}_{n+1}.    
  \end{equation}
  Let $y \in S_V$ be any point. Using \eqref{e:conversion of cube and point coordinates}, the cube $C_{a_{n+1}}^d(t_1, \dots, t_d) \in \mathcal{C}_{a_{n+1}}$ containing $y$ satisfies $|\{i \le d : \{t_ia_n\} \in [\frac{9}{20}, \frac{13}{20})\}|\le k$. Using also that $S_V \subseteq K_n$ and that $V \cap \bigcup \mathcal{S}_{n+1} = \emptyset$, we obtained that $y \in K_n \setminus \bigcup \mathcal{S}_{n+1}$, hence it is in $\mathcal{T}_{n+1}$.
  
  For each cube $C \in \mathcal{T}_{n+1}$, the event of $C$ being chosen into $\mathcal{C}_{n+1}$ only depends on the events $\{C \in \mathcal{T}_{n+1}^B\}$ and $\{\{C' \in \mathcal{T}_{n+1}^B \} : C' \in \mathcal{E}_C\}$, where
  $$\mathcal{E}_C = \left\{C' \in \mathcal{T}_{n+1} :  \frac{C+C'}{2} \cap \interior C_{n+1} \neq \emptyset\right\}.$$
  Since these are independent events each with probability $\frac{1}{2}$ and $|\mathcal{E}_C| \le 3^d$ for each $C \in \mathcal{T}_{n+1}$, \begin{equation}\label{e:p def}
   \mathbb{P}(C \in \mathcal{C}_{n+1}) \ge p, \text{ where } p=\frac{1}{2^{3^d + 1}}.    
  \end{equation}
  Moreover, if for the cubes $C = C^d_{a_{n+1}}(t_1, \dots, t_d), C' = C^d_{a_{n+1}}(t_1', \dots, t_d') \in \mathcal{T}_{n+1}$ there is an $i$ such that 
  \begin{equation}\label{e:neighboring cubes distinct}
    \textstyle |t_i - t_i'| \ge 3,
  \end{equation}
  then $\mathcal{E}_C \cap \mathcal{E}_{C'} = \emptyset$.  
  
  Let $i \le d$ be a coordinate such that $\length(\pr_i(S_V)) \ge \frac{1}{10\sqrt{d}a_n}$, such a coordinate exists. Then, using \eqref{e:S_V subset T}, we can choose a family $\mathcal{F}_V \subseteq \mathcal{T}_{n+1}$ with $|\mathcal{F}_V| \ge \lfloor\frac{a_{n+1}}{10\sqrt{d}a_n}\rfloor$, where $\lfloor x \rfloor$ denotes the integer part of a non-negative real number $x$, such that $C, C' \in \mathcal{F}_V$, $C \neq C'$ imply $\pr_i(C) \neq \pr_i(C')$ and $C \cap S_V \neq \emptyset$. 
  Let
  \begin{align*}
      \mathcal{F}_V^- = \{C \in \mathcal{F}_V : \min(\pr_i(C)) < \min(\pr_i(C_{n+1}))\}, \\ 
      \mathcal{F}_V^+ = \{C \in \mathcal{F}_V : \min(\pr_i(C)) > \min(\pr_i(C_{n+1}))\}.
  \end{align*}
  Since at most one cube $C$ from $\mathcal{F}_V$ satisfies $\pr_i(C) = \pr_i(C_{n+1})$, we clearly have $|\mathcal{F}_V^-| + |\mathcal{F}_V^+| \ge |\mathcal{F}_V|-1$. Let us pick the larger of these two sets as $\mathcal{F}_V'$, then $|\mathcal{F}_V'| \ge \lfloor \frac{|\mathcal{F}_V|}{2}\rfloor$, therefore $|\mathcal{F}_V'| \ge \lfloor\frac{a_{n+1}}{20\sqrt{d}a_n}\rfloor$. 
  Let us choose every third cube from $\mathcal{F}_V'$ into $\mathcal{F}_V''$ so that $|\mathcal{F}_V''| \ge \lfloor\frac{a_{n+1}}{60\sqrt{d}a_n}\rfloor$, and the cubes in $\mathcal{F}_V''$ satisfy \eqref{e:neighboring cubes distinct}. 
  
  We claim that the events $$\{\{C \in \mathcal{C}_{n+1} \} : C \in \mathcal{F}_V''\}$$ are independent. In order to prove our claim, it is enough to show that the system $\{\{C\} \cup \mathcal{E}_C : C \in \mathcal{F}_V''\}$ consists of pairwise disjoint sets. Let $C, C' \in \mathcal{F}_V''$. \eqref{e:neighboring cubes distinct} makes sure that $\mathcal{E}_C \cap \mathcal{E}_{C'} = \emptyset$. To show that $C \not \in \mathcal{E}_{C'}$, (and similarly, $C' \not \in \mathcal{E}_{C}$), one needs to show that $\frac{C \cap C'}{2} \cap \interior C_{n+1} = \emptyset$, a fact that easily follows from $\mathcal{F}_V''$ being a subset of either $\mathcal{F}_V^-$ or $\mathcal{F}_V^+$. 
  
  Using this claim and \eqref{e:p def}, it follows that the probability that no cube from $\mathcal{F}_V''$ is selected is at most $(1-p)^{\lfloor \frac{a_{n+1}}{60\sqrt{d}a_n}\rfloor}$. 
  
 Consider an arbitrary $k$-dimensional subspace $V$ that meets $[0, 1]^d$. If $V \cap\bigcup  \mathcal{S}_{n+1} = \emptyset$, then let $ L_V$ denote the line containing the segment $S_V$; otherwise (i.e.~when $V \cap\bigcup \mathcal{S}_{n+1} \neq \emptyset$) choose an arbitrary line  $ L_V \subseteq V$ such that $ L_V \cap\bigcup \mathcal{S}_{n+1} \neq \emptyset$.  In either case, 
 $$\mathbb{P}( L_V \cap K_{n+1} = \emptyset) \le (1-p)^{\left\lfloor \frac{a_{n+1}}{60\sqrt{d}a_n}\right\rfloor},$$
  since this is obvious if $L_V \cap \bigcup \mathcal{S}_{n+1} \neq \emptyset$, and otherwise it is clear that $L_V \cap K_{n+1} = \emptyset$ implies that no cube from $\mathcal{C}$ is selected. 
  
  Using \autoref{cor:number of line traces}, the number of line traces in $\mathcal{K}^d_{a_{n+1}}$ associated to lines of the form $ L_V$ is at most $(2a_{n+1})^{\frac{5}{2}d!}$. It follows that 
  $$\mathbb{P}\big(\exists V (V \cap K_{n+1}) = \emptyset\big) \le \mathbb{P}\big(\exists V ( L_V \cap K_{n+1}) = \emptyset\big) \le (2a_{n+1})^{\frac{5}{2}d!}(1-p)^{\left\lfloor \frac{a_{n+1}}{60\sqrt{d}a_n}\right\rfloor}.$$
  Since $p$ only depends on $d$ by \eqref{e:p def}, it is easy to check that the above quantity tends to $0$ as $a_{n+1} \to \infty$. Hence, for large enough $a_{n+1}$ it is less than $1$, therefore with positive probability the random construction yields a compact set $K_{n+1}$ satisfying (iii). 
\end{proof}

We conclude this section by proving \autoref{l:subspace intersects skeleton or center}.

\begin{restate}{l:subspace intersects skeleton or center}
  Suppose that $V$ is a $k$-dimensional affine subspace of $\R^d$ for some $k \ge \frac{d}{2}$ that intersects $[0, 1]^d$. Then $V$ intersects either the set 
  $$
    \textstyle S = \left\{x \in [0, 1]^d : \left|\left\{i \le d : x_i \in \left[0, \frac{1}{20}\right] \cup \left[\frac{19}{20}, 1\right]\right\}\right| \ge k + 1\right\}, 
  $$
  or the set 
  $$
    \textstyle T = \left\{x \in \left[\frac{1}{20}, \frac{19}{20}\right]^d : \left|\left\{i \le d : x_i \in \left[\frac{8}{20}, \frac{14}{20}\right]\right\}\right| \le k \right\}.
  $$
\end{restate}
\begin{proof}
  We first show that $V$ intersects either $S$ or $\left[\frac{1}{20}, \frac{19}{20} \right]^d$; in the second case it will be easy to find an intersection point that is actually in $T$. 
  
  For any point $x \in [0, 1]^d \cap V$ let $I(x) = \{ i \le d : x_i \in \left[0, \frac{1}{20}\right] \cup \left[\frac{19}{20}, 1\right]\}$. Let us fix
  \begin{equation}
  \label{e:x_0 def}
  \text{$x^* \in [0, 1]^d \cap V$ with $|I(x^*)|$ maximal among points from $[0, 1]^d \cap V$.}
  \end{equation}
  Without loss of generality, we can assume that $x^*_i \in \left[0, \frac{1}{20}\right]$ for each $i \in I(x^*)$ (although $T$ is not symmetric, we are now only proving the existence of a point in $S \cup [\frac{1}{20}, \frac{19}{20}]^d$). If $|I(x^*)| \ge k +1$, then we are done, $x^* \in S \cap V$. Otherwise, let $U$ be the linear subspace of $\R^d$ generated by the coordinate vectors corresponding to coordinates in $I(x^*)$. Then $\dim U = |I(x^*)| \le k$. 
  
  Recall that $\pr_U\colon \R^d \to U$ denotes the orthogonal projection onto the affine (in fact, linear) subspace $U$. If $\pr_U \restriction V$ is not injective, then there is a line $L$ orthogonal to $U$ such that $x^* \in L \subseteq V$. Then for each $y \in L$, $\pr_U(y) = \pr_U(x^*)$, hence $y_i = x^*_i$ for each $i \in I(x^*)$. If $y$ is chosen to be a boundary point of the set $L \setminus [0, 1]^d$ in the space $L$, then it will have a coordinate $i$ with $i \not \in I(x^*)$, but $y_i \in \left[0, \frac{1}{20}\right] \cup \left[\frac{19}{20}, 1\right]$, therefore $|I(y)| > |I(x^*)|$ contradicting \eqref{e:x_0 def}. It follows that $\pr_U \restriction V$ is injective. 
  
  Since $\dim U \le k$ and $\dim V = k$, the injectivity of $\pr_U \restriction V$ implies that it is a bijection. Let $z = (\pr_U\restriction V)^{-1}\left(\frac{1}{20}, \frac{1}{20}, \dots, \frac{1}{20}\right)$. We now show that $z \in \left[\frac{1}{20}, \frac{19}{20}\right]^d$. 
  
  Clearly, for $i \in I(x^*)$, $z_i = \frac{1}{20}$. Consider the affine map
  \[f\colon [0,1]\to [x^*, z],\qquad f(p)=(1-p)x^*+pz,\]
  where $[x^*, z]$ is the segment in $\R^d$ determined by the points $x^*$ and $z$. Let
  $$B = \left\{p \in [0, 1] : f(p)_j \in \left[\frac{1}{20}, \frac{19}{20}\right]
 \text{ for each }j \not \in I(x^*)\right\}.$$
  Note that $0 \in B$ and if $1 \in B$ as well, then we are done, $z \in \left[\frac{1}{20}, \frac{19}{20}\right]^d$. If $1 \not\in B$, let $p = \sup\{q \in [0, 1] : [0, q] \subseteq B\}$.  Using that $f(p)$ is on the segment determined by $x^*$ and $z$, $f(p)_i \in \left[0, \frac{1}{20}\right]$ for each $i \in I(x^*)$. Since $B$ is closed, $p \in B$, hence $f(p)_j \in \left[\frac{1}{20}, \frac{19}{20}\right]$ for each $j \not \in I(x^*)$. Since each neighborhood of $p$ contains a point outside $B$, there has to be a coordinate $j \not \in I(x^*)$ with $f(p)_j \in \left\{\frac{1}{20}, \frac{19}{20}\right\}$. It follows that $|I(f(p))| \ge k + 1$, contradicting \eqref{e:x_0 def}. 
  
  The above argument yields that indeed, $z$ must be in $V \cap \left[\frac{1}{20}, \frac{19}{20}\right]^d$. For a point $x \in V$ let $J(x) = \left\{i \le d : x_i \in \left[\frac{8}{20}, \frac{14}{20}\right]\right\}$. Let $z^* \in V \cap \left[\frac{1}{20}, \frac{19}{20}\right]^d$ with $|J(z^*)|$ minimal among such points. We finish the proof of the lemma by showing that $|J(z^*)| \le k$.
  
  Suppose towards the contrary that $|J(z^*)| \ge k + 1$. Since $k \ge \frac{d}{2}$, we obtain $|\{1, 2, \dots, d\} \setminus J(z^*)| < k$. Let $W$ be the linear subspace generated by the coordinate vectors corresponding to coordinates in $\{1, 2, \dots, d\} \setminus J(z^*)$, which clearly satisfies that $\dim W < k$. Since $\dim V = k$, the map $\pr_W \restriction V$ is not injective. It follows that there is a line $L' \subseteq V$ which is orthogonal to $W$ and passes through $z^*$. Using the same idea as before, one can find a point $w \in L'$ with $w \in V \cap \left[\frac{1}{20}, \frac{19}{20}\right]^d$, $w_i = z^*_i$ for each $i \not\in J(z^*)$, and $w_j \in \left[\frac{1}{20}, \frac{19}{20}\right] \setminus \left[\frac{8}{20}, \frac{14}{20}\right]$ for some $j \in J(z^*)$. It follows that $|J(w)| < |J(z^*)|$, contradicting the minimality of $|J(z^*)|$. 
\end{proof}

\section{\texorpdfstring{$\ell$-fold sumsets}{l-fold sumsets}}
\label{sec:multiple}

In this section we extend our scope to $\ell$-fold sumsets, and improve \autoref{t:K+K construction}.

\begin{theorem}\label{t:desklegs}
$\exists (\ell, d-1, d)$-set $\implies$ $\exists (\ell+1, d, d+1)$-set, that is, if there exists a compact set $K \subseteq [0,1]^d$ that intersects each $(d-1)$-dimensional affine subspace (i.e.~hyperplane) that meets $[0,1]^d$ and the $\ell$-fold sumset $\ell K$ is nowhere dense, then there exists a compact set $K' \subseteq [0,1]^{d+1}$ that intersects each $d$-dimensional affine subspace (i.e.~hyperplane) that meets $[0,1]^{d+1}$ and the $(\ell+1)$-fold sumset $(\ell+1)K'$ is nowhere dense.
%\[  \underbrace{K'+ \dots + K'}_{\ell+1 \text{ %times}} \text{ is nowhere dense.} \]
\end{theorem}

\begin{proof}
Let 
\[ K' = (K \times \{0\}) \cup \bigcup_{x \in \{0,1\}^d} \{ x \} \times [0,1]. \]
Now we will check that $(\ell+1)K'$
is nowhere dense, and $K'$ intersects each $d$-dimensional affine subspace that meets $[0,1]^{d+1}$.
% From now on with a slight abuse of notation we identify $K$ with its embedding into the $d+1$-dimensional euclidean space.
As $K\subseteq [0,1]^d$ and the hyperplane $\{x\in\R^d : x_1 + \ldots + x_d = 0\}$ intersects $[0,1]^d$ only at the origin, $K$ must contain the origin and therefore $iK \subseteq j K$ for $i \leq j$.
 Observe that this implies the following
\[ (\ell+1)K' \subseteq (\ell+1) (K\times \{0\}) \ \cup \ (\ell K + \{0,1,\dots, \ell+1\}^d) \times [0,\ell+1].  \]
Now $(\ell+1) (K\times \{0\})  \subseteq \mathbb{R}^d \times \{0\}$, hence nowhere dense in $\mathbb{R}^{d+1}$, and as $\ell K$ was nowhere dense in $\mathbb{R}^d$ so is the union of its finitely many translates $(\ell K + \{0,1,\dots, \ell+1\}^d)$, therefore  $(\ell K  + \{0,1,\dots, \ell+1\}^d) \times [0,\ell+1]$ is nowhere dense in $\mathbb{R}^{d+1}$, so $(\ell +1)K'$ is indeed nowhere dense.

It remains to show that $K'$ intersects every hyperplane $V\subseteq \mathbb{R}^{d+1}$ with $V \cap [0,1]^{d+1} \neq \emptyset$. Fix such a hyperplane $V$, a normal vector $v$, and $c \in \mathbb{R}$ with 
\begin{equation} \label{egyenl}V = \{ w \in \mathbb{R}^{d+1}: \ v\cdot w = c \},
\end{equation}
where $v\cdot w$ denotes the scalar product of $v$ and $w$.

If $V \cap \left([0,1]^d \times \{0\}\right) \neq \emptyset$, then  $V \cap (\mathbb{R}^d \times \{0 \})$ is an at least $(d-1)$-dimensional affine subspace that intersects $[0,1]^d \times \{0\}$, hence $K \times \{0\}$ intersects it by our assumptions.

Otherwise $V \cap \left([0,1]^d \times \{0\}\right) = \emptyset$, then it is straightforward to check (using $V \cap [0,1]^{d+1} \neq \emptyset$) that $v_{d+1} \neq 0$. For a vector $x\in\R^d$ consider the unique number $f(x)\in\R$ which satisfies that $(x_1, x_2, \dots, x_d, f(x)) \in V$ (i.e.~the solution of $\eqref{egyenl}$). It is clear that $f\colon [0,1]^d \to \mathbb{R}$ is a well-defined affine function.

Now $V \cap \left([0,1]^d \times \{0\}\right) = \emptyset$ implies that $0$ is not in the range of $f$, hence $f$ takes only strictly positive values. We only have to show that there exists $x \in \{0,1\}^d$ with $f(x) \in [0,1]$. But otherwise $x \in \{0,1\}^d$ implies $f(x)>1$, and as $f$ is an affine function $f$ would map $[0,1]^d$ into $(1,\infty)$, contradicting $V \cap [0,1]^{d+1} \neq \emptyset$.
\end{proof}

A straightforward induction starting with the case $d=2$ of Theorem \ref{t:K+K construction} yields the following.

\begin{corollary} \label{c:ex}
$\forall d \ge 2 \ \exists (d, d-1, d)$-set, that is,  a compact set $K \subseteq [0, 1]^d$ such that $K$ intersects each $(d-1)$-dimensional affine subspace that meets $[0,1]^d$, but the $d$-fold sumset $dK$ is nowhere dense.
\end{corollary}

We study another similar construction in the following theorem, which at first sight seems to be less powerful than \autoref{t:desklegs}. However, this construction works in a slightly more general setting, and we will also use it to generalize the main result of Section \ref{sec:self-similar}.

\begin{theorem}\label{t:product with interval}
$K$ is an $(\ell, k, d)$-set $\implies K \times [0, 1]$ is an $(\ell, k+1, d+1)$-set, that is, if $K \subseteq [0, 1]^d$ is a compact set intersecting each $k$-dimensional affine subspace that meets $[0,1]^d$ and the $\ell$-fold sumset $\ell K$ is nowhere dense, then $K \times [0, 1]$ intersects each $(k+1)$-dimensional affine subspace that meets $[0,1]^{d+1}$, but the $\ell$-fold sumset $\ell(K\times [0, 1])$ is nowhere dense.
\end{theorem}

\begin{proof}
$\ell(K\times [0, 1]) = \ell K\times [0, \ell]$, from which it is easy to see that $\ell(K\times [0, 1])$ is nowhere dense. Hence it suffices to check that $\ell(K\times [0, 1])$ intersects each $(k+1)$-dimensional affine subspace that meets $[0,1]^{d+1}$. Let $V$ be such a subspace, and fix $x \in V \cap [0,1]^{d+1}$. Consider the hyperplane $H = \R^d \times \{x_{d+1}\}$. Then $H \cap V$ is an affine subspace of dimension at least $k$ (as $H$ is of co-dimension $1$), and since $H \cap (K \times [0, 1]) = K \times \{x_{d+1}\}$ we obtain that $H \cap V$ intersects $K \times \{x_{d+1}\}$, hence $V$ also intersects $K\times [0, 1]$.
\end{proof}

\section{A non-existence result}
\label{sec:negative}

\begin{theorem} \label{t:nonex}
If there exists an $ (\ell,k,d)$-set then $k>\frac{\ell-1}{\ell}(d-1)$, that is, if there is a compact set $K \subseteq [0,1]^d$ that intersects each $k$-dimensional affine subspace that meets $[0,1]^d$ and the $\ell$-fold sumset $\ell K$ is nowhere dense, then $k>\frac{\ell-1}{\ell}(d-1)$.
\end{theorem}

\begin{proof} Assume to the contrary that there is an $(\ell,k,d)$-set $K$ so that $k\leq \frac{\ell-1}{\ell}(d-1)$. We will obtain a contradiction by proving that $\ell K$ has non-empty interior.

First we prove that the $(d-k-1)$-dimensional faces of $[0,1]^d$ are subsets of $K$. Let $z_1,\dots,z_{d-k-1}\in [0,1]$ be given and let $z=(z_1,\dots,z_{d-k-1},0,\dots,0)\in [0,1]^d$. By symmetry it is enough to prove that there is a $k$-dimensional affine subspace $V$ such that $V\cap [0,1]^d=\{z\}$. Indeed, let 
\begin{equation*}
V=\{x\in \R^d : x_i=z_i \text{ for all } 1\leq i\leq d-k-1, \text{ and } x_{d-k}+\dots+ x_d=0\}.    
\end{equation*}
Clearly $V$ is a $k$-dimensional affine subspace and $z\in V$. If $x\in V\cap [0,1]^d$ then $x_i\geq 0$ for all $i\geq d-k$ and $x_{d-k}+\dots+ x_d=0$, so $x_i= 0$ for all $i\geq d-k$. Thus $V\cap [0,1]^d=\{z\}$.

Let $s=(\ell-1)(d-k-1)$. First assume that $s\geq d$. As the $(d-k-1)$-dimensional faces of $[0,1]^d$ are subsets of $K$, we obtain that $[0,1]^d \subseteq (\ell-1)K\subseteq \ell K$, and we are done. Hence we may assume that $s<d$. Similarly as above we have 
\begin{equation} \label{e1}
[0,1]^s\times \{0\}^{d-s}\subseteq (\ell-1)K. 
\end{equation}
It is clear that our assumption $k\leq \frac{\ell-1}{\ell}(d-1)$ is equivalent to $s\geq k$, hence for every $y\in [0,1]^{d-s}$ the affine subspace $\R^s\times \{y\}$ intersects $K$. Let us write $[0,1]^s$ as the union of $2^s$ closed dyadic cubes $Q_1,\dots, Q_{2^s}$ of edge length $1/2$, and for each $1\leq i\leq 2^s$ let
\begin{equation*}
A_i=\{y\in  [0,1]^{d-s}: K\cap (Q_i\times \{y\})\neq \emptyset  \}.
\end{equation*} 
As the $A_i$ are compact sets such that $\bigcup_{i=1}^{2^s} A_i=[0,1]^{d-s}$, at least one of them has a non-empty interior. We may assume without loss of generality that $Q_1=[0,1/2]^s$ and there is a non-empty ball $B\subseteq [0,1]^{d-s}$ such that $B\subseteq A_1$. That is,  
\begin{equation} \label{e2}
\forall y\in B ~ \exists x\in [0,1/2]^s \text{ such that } (x,y)\in K.
\end{equation}
Finally, we prove that  
\begin{equation*}
\textstyle\left[\frac12,1\right]^{s}\times B\subseteq (\ell-1)K+K=\ell K,
\end{equation*}
which implies that $\ell K$ has non-empty interior, hence the proof will be complete. 
 Let $(x,y)\in \left[\frac12,1\right]^{s}\times B $, by \eqref{e2} there is an $x'\in [0,1/2]^s$ such that $(x',y)\in K$. By \eqref{e1} we have $(x-x',0)\in (\ell-1)K$ (here $0$ denotes the origin of $\R^{d-s}$).  Therefore, $(x,y)=(x-x',0)+(x',y)\in  (\ell-1)K+K$.
\end{proof}

Theorems~\ref{t:K+K construction} and \ref{t:nonex} imply the following.

\begin{corollary}
\label{c:(2,k,d) <=> k >= d/2}
The following are equivalent:
\begin{enumerate}[label=\normalfont{(\roman*)}]
\item $k \ge \frac{d}2$,
\item $\exists (2, k, d)$-set, that is, a compact set $K \subseteq [0, 1]^d$ such that $K$ intersects each $k$-dimensional affine subspace that meets $[0,1]^d$, but $K + K$ is nowhere dense.
\end{enumerate}
\end{corollary}

Corollary~\ref{c:ex} and Theorem~ \ref{t:nonex} yield the following.

\begin{corollary}
The following are equivalent:
\begin{enumerate}[label=\normalfont{(\roman*)}]
\item $k = d-1$,
\item $\exists (d, k, d)$-set, that is, a compact set $K \subseteq [0, 1]^d$ such that $K$ intersects each $k$-dimensional affine subspace that meets $[0,1]^d$, but the $d$-fold sumset $dK$ is nowhere dense,
\item $\exists (d-1, k, d)$-set, that is, a compact set $K \subseteq [0, 1]^d$ such that $K$ intersects each $k$-dimensional affine subspace that meets $[0,1]^d$, but the $(d-1)$-fold sumset $(d-1)K$ is nowhere dense.
\end{enumerate}
\end{corollary}

\section{Generic compact sets}
\label{sec:generic}

In this section we prove that if there exists an $(\ell, k, d)$-set, then in a reasonable complete metric space of compact sets the generic element is also an $(\ell, k, d)$-set. 

\begin{definition}
Let $(\mathcal{K}([0,1]^d),d_{H})$ be the set of non-empty compact subsets of
$[0,1]^d$ endowed with the \emph{Hausdorff metric};
that is, for each $K_1,K_2\in \mathcal{K}([0,1]^d)$,
%%%%%%%%%%%%%%%%
\[d_\mathrm{H}(K_1,K_2)=\max\left\{\max_{x_1\in K_1}\min_{x_2\in K_2}d(x_1,x_2),\,\max_{x_2\in K_2}\min_{x_1\in K_1}d(x_1,x_2)\right\}.\]
\end{definition} 

\begin{definition} For $0\leq k<d$ let 
\begin{align*}
\mathcal{K}_k([0,1]^d) = 
\{ & K \in \mathcal{K}([0,1]^d): K \text{ intersects each $k$-dimensional affine subspace} \\ & \text{that meets } [0,1]^d \}.
%\text{every $k$-dimensional affine subspace $V \subseteq \mathbb{R}^d$}\\&\qquad\text{satisfies that if }[0,1]^d \cap V\neq \emptyset\text{, then } K \cap V \neq \emptyset\}.
\end{align*}
\end{definition}
Now it is straightforward to verify that $\mathcal{K}_k([0,1]^d)$ is a closed subset in $\mathcal{K}([0,1]^d)$, therefore it is a complete metric space.

\begin{theorem} \label{thgeneric}
$\exists (\ell, k, d)$-set $\implies$ the generic element of $\mathcal{K}_k([0,1]^d)$ is an $(\ell, k, d)$-set, that is, if there exists
$K_0 \in \mathcal{K}_k([0,1]^d)$ such that the $\ell$-fold sumset $\ell K_0$ is nowhere dense, then $\ell K$ is nowhere dense
for comeager many $K \in \mathcal{K}_k([0,1]^d)$.
\end{theorem}

\begin{proof}
 Fixing a countable base $(B_i)_{i \in \mathbb{N}}$ of $[0,1]^d$ it suffices to prove that for each $i$ the set
 \[ U_i = \{K \in \mathcal{K}_k([0,1]^d): \ \ell K \nsupseteq B_i \} \]
 is dense open.
 
 If $K \in \mathcal{K}_k([0,1]^d)$, $x \in B_i$ are such that 
 \begin{equation} \label{xnem} x \notin \ell K, \end{equation}
 then there is a neighbourhood of $K$ in $\mathcal{K}([0,1]^d)$ satisfying $\eqref{xnem}$ (hence also in $\mathcal{K}_k([0,1]^d)$).
 
 For the density of $U_i$, fix $K \in \mathcal{K}_k([0,1]^d)$, and $\varepsilon >0$. We will construct $K' \in \mathcal{K}_k([0,1]^d)$ with $\ell K'$ nowhere dense and $d_\mathrm{H}(K,K') < \varepsilon$. Fix $n>0$ such that $\frac{2\sqrt{d}}{n} < \varepsilon$. Define the finite set 
 $$\textstyle H = \left\{ x \in \left\{\frac{m}{n} : \ m \in \mathbb{Z}, \ 0 \leq m < n\right\}^d: \ \left( x + \left[0,\frac{1}{n}\right]^d\right) \cap K  \neq \emptyset \right\}. $$
 Now obviously  $$\textstyle H + \left[0,\frac{1}{n}\right]^d \supseteq K,$$ thus $H + \left[0,\frac{1}{n}\right]^d \in \mathcal{K}_k([0,1]^d)$, and also $d_\mathrm{H}\big(K,H + \left[0,\frac{1}{n}\right]^d\big) \leq \frac{\sqrt{d}}{n}$.
 
We know that $H + \left[0,\frac{1}{n}\right]^d \in \mathcal{K}_k([0,1]^d)$, that is, this set intersects each $k$-dimensional affine subspace intersecting $[0,1]^d$. But we also know that $K_0\in \mathcal{K}_k([0,1]^d)$, which implies that for any translation vector $h\in \R^d$, the set $h+\frac{1}{n} {\cdot} K_0$ intersects the same $k$-dimensional affine subspaces as the set $h+\left[0,\frac1n\right]^d$. These imply that $K' = H + \frac{1}{n} {\cdot} K_0 \in \mathcal{K}_k([0,1]^d)$.
 
 It can be easily seen that 
 $d_{\mathrm{H}}(K',H + \left[0,\frac{1}{n}\right]^d) \leq \frac{\sqrt{d}}{n}$, thus recalling $d_{\mathrm{H}}(K,H + \left[0,\frac{1}{n}\right]^d) \leq \frac{\sqrt{d}}{n}$  we obtain 
 \[ d_\mathrm{H}\left(K' ,K \right) \leq \frac{2\sqrt{d}}{n}< \varepsilon. \]
 Finally, $\ell K' = \ell H + \ell\left(\frac{1}{n} {\cdot} K_0\right)$ is nowhere dense because $H$ is finite and the sum $\ell K_0$ is nowhere dense.
\end{proof}

\section{A self-similar construction}
\label{sec:self-similar}

In this section we prove the existence of a $(2,1,2)$-set in a constructive fashion, without using the probabilistic method described in \autoref{sec:full projections}. Additionally, the $(2,1,2)$-set constructed in this section will be self-similar. However, this proof cannot be easily generalized for $d>2$, as it relies on some special properties of a certain planar pattern, which we found with the combination of computer-aided search and manual improvements.

\begin{theorem}\label{self-similar}
There exists a self-similar $(2,1,2)$-set, that is, a self-similar, compact set $K\subseteq [0,1]^2$ such that $K$ intersects each line that meets $[0, 1]^2$, but $K+K$ is nowhere dense.
\end{theorem}

\begin{proof}
The fractal $K$ will be the intersection of the sets $K_0\supseteq K_1\supseteq K_2\supseteq\ldots $, which are defined by a recursion of the form
\[K_0=[0,1]^2\quad\text{and}\quad K_{i+1}=\bigcup_{t\in T} \varphi_t(K_i)\quad\text{for }i\in\mathbb{N}.\]
\autoref{fig:pict} illustrates the first two steps of the construction.
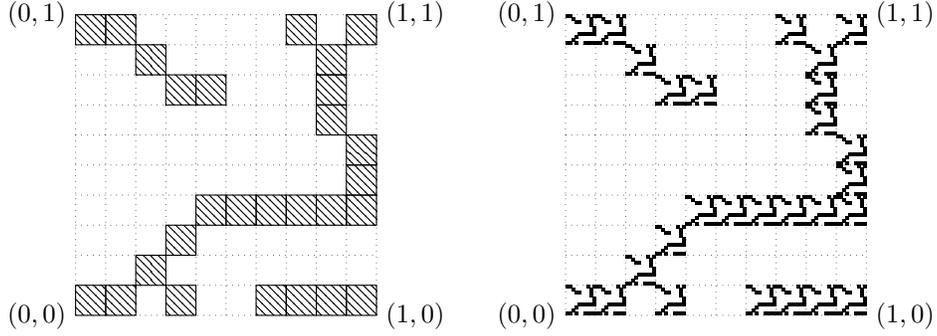
\begin{figure}[htb]
\begin{tikzpicture}[x=0.4cm, y=0.4cm]
\squaregrid
\foreach \pos in \plaincoords {
  \draw[pattern=north west lines] \pos rectangle +(1,1);
}
\end{tikzpicture}
\quad
\begin{tikzpicture}[x=0.4cm, y=0.4cm]
\squaregrid
\foreach \pos in \plaincoords {
  %\draw \pos rectangle +(1,1);
  \foreach \subpos in \plaincoords {
    \edef\mytx{($\pos + 1/10*\subpos$) rectangle +(0.1,0.1);}
    \expandafter\fill\mytx 
}
}
\end{tikzpicture}
\caption{\label{fig:pict}The sets $K_1$ (union of the \arabic{coordcount} hatched squares from the grid $\mathcal{C}^2_{10}$) and $K_2$ (union of the $\arabic{coordcount}^2$ filled squares from $\mathcal{C}^2_{100}$)}
\end{figure}

For a formal definition of this fractal, we use an iterated function system $\{\varphi_t : t\in T\}$ that consists of \arabic{coordcount} contractions, each of similarity ratio $\frac1{10}$. For a translation vector $t\in\R^2$ we define the contraction $\varphi_t$ as
\[\varphi_t: \R^2\to\R^2,\qquad \varphi_t(x)=\frac1{10}(x+t);\]
and we construct the fractal using the translation vectors in
\begin{align*}
T=\{&\coords{\\&}\}\subseteq \R^2.
\end{align*}
It is easy to see that this definition indeed implies that $K_0\supseteq K_1\supseteq K_2\supseteq\ldots $ and the intersection $K = \bigcap_{i\in\mathbb{N}}K_i$ will be a compact subset of $[0,1]^2$. 

%\begin{remark}
%This construction is very simple compared to the random patterns produced in \autoref{t:K+K construction}: there each step reduces the size of the grid cubes by a factor of at least $10^6$ (even in the simple case of $d=2$).
%\end{remark}

%The rest of the proof consists of \autoref{klarge} (which shows that $K$ intersects every line intersecting the unit square) and \autoref{ksmall} (which shows that $K+K$ is nowhere dense). 

\begin{fact}\label{k1large}
If a line $L$ intersects $[0,1]^2$, then it also intersects the set $K_1$.
\end{fact}

\begin{proof}
It is easy to verify this fact on \autoref{fig:kone}, which shows the set $K_1$ in the unit square.

\begin{figure}[htb]
\begin{tikzpicture}[x=0.5cm, y=0.5cm]
\squaregrid
\foreach \x in \plaincoords
  \draw[pattern=north west lines] \x rectangle +(1,1);
\draw (2,-1) -- (12,4);
\fill (4,0) circle[radius=0.3ex] node[anchor=north west] {$a$} (6,1) circle[radius=0.3ex] node[anchor=south east] {$b$} (10,3) circle[radius=0.3ex] node[anchor=north west] {$c$};
\draw (0,-0.5) -- (7.5,10.75);
\fill (1,1) circle[radius=0.3ex] node[anchor=south east] {$a'$} (5,7) circle[radius=0.3ex] node[anchor=north west] {$b'$} (7,10) circle[radius=0.3ex] node[anchor=south east] {$c'$};
\end{tikzpicture}
\caption{\label{fig:kone}The set $K_1$ (union of the \arabic{coordcount} hatched squares) in $[0,1]^2$}
\end{figure}
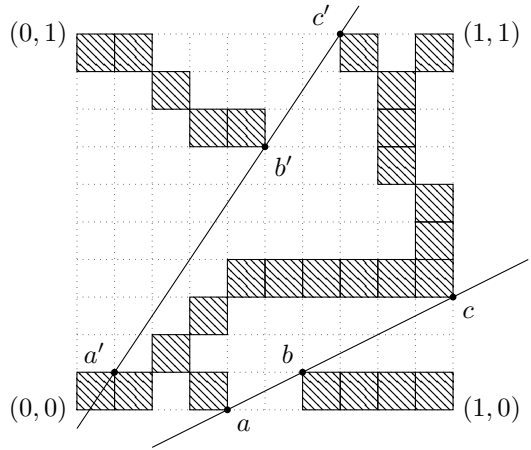

For a more formal proof, observe that if the line $L$ intersects the diagonal between $(0,0)$ and $(1,1)$, then it intersects $K_1$ because these corners are connected by a continuous path in $K_1$. Otherwise there are two possible cases:

\textbf{Case 1:} $L$ intersects the bottom and right sides of the square $[0,1]^2$.

Notice that the set $K_1\cup [a, b]$ (where $[a, b]$ denotes the closed segment between the points $a$ and $b$ marked on \autoref{fig:kone}) contains a continuous path between $(0,0)$ and $(1,0)$, while $K_1\cup [b,c]$ contains a continuous path between $(1,0)$ and $(1,1)$. These mean that $L$ must intersect both of these sets, but then either $L$ intersects $K_1$ or $L$ intersects both of the collinear segments $[a,b]$ and $[b,c]$. The latter means that $L$ passes through the common endpoint $b$, but then $b\in K_1$ allows us to conclude this case.

\textbf{Case 2:} $L$ intersects the top and left sides of the square $[0,1]^2$.

This case can be handled analogously to Case 1, using the collinear points $a'$, $b'$ and $c'$ instead of $a$, $b$ and $c$.
\end{proof}

\begin{claim}\label{klarge}
If a line $L$ intersects $[0,1]^2$, then it also intersects $K$.
\end{claim}

\begin{proof}
For a line $L$, the set $L\cap K$ can be written as the intersection of the compact sets \[L\cap K_0 \supseteq L\cap K_1 \supseteq L\cap K_2 \supseteq\ldots\]
and this means that it is enough to prove that these sets are all non-empty.

We use induction to show that each index $i\in\mathbb{N}$ satisfies that if $L\cap [0,1]^2\neq \emptyset$, then $L\cap K_i\neq \emptyset$. The case of $i=0$ is trivial. For the induction step, assume that the induction hypothesis holds for some $i\in\mathbb{N}$.

According to \autoref{k1large}, our line $L$ intersects $K_1$ and this implies that $L$ intersects $\varphi_{t_0}(K_0)$ for a (not necessarily unique) $t_0\in T$. Then the line $\varphi_{t_0}^{-1}(L)$ intersects $K_0$ and so the induction hypothesis implies that the intersection \(\varphi_{t_0}^{-1}(L)\cap K_i\) is non-empty.
But then
\[\emptyset \neq\varphi_{t_0}(\varphi_{t_0}^{-1}(L)\cap K_i) = L\cap\varphi_{t_0}(K_i)\subseteq L\cap K_{i+1},\]
and this shows the induction hypothesis also holds for $i+1$.
\end{proof}

\begin{fact}\label{k1small}
There is a vector $v\in\R^2$ such that $v-K_1$ is disjoint from $K_1 + \Z^2$.
\end{fact}

%Note that $K_1 + \Z^2$ is the periodic set where the pattern $K_1$ is repeated in every cell of the Cartesian grid; and $v-K_1$ is the reflection of $K_1$ across the point $\frac{v}{2}$.

\begin{proof}
We will choose $v=(1.45, 1.55)$. If $(a, b)\in\Z^2$ such that either $a\notin\{0,1\}$ or $b\notin\{0,1\}$, then $K_1\subseteq [0,1]^2$ immediately implies that $v-K_1$ is disjoint from $K_1+(a,b)$. Therefore it is sufficient to verify that
\[v-K_1\quad\text{is disjoint from}\quad K_1+\{(0,0), (1,0), (0,1), (1,1)\}.\]

Note that $-K_1$ is the reflection of $K_1$ with respect to the origin and $v-K_1$ is the translation of $-K_1$ by the vector $v$; while the set $K_1+\{(0,0), (1,0), (0,1), (1,1)\}$ is the union of four translated copies of $K_1$. We illustrate these planar shapes on \autoref{fig:dblsize}, and here it is clearly visible that $v-K_1$ (the filled area) is disjoint from the four translated copies of $K_1$.

\begin{figure}[htb]
\begin{tikzpicture}[x=0.3cm, y=0.3cm]
\squaregrid[2]
\foreach \quad in {(0, 0)} {
  \foreach \x in \plaincoords
    \draw[pattern=crosshatch] \quad ++\x rectangle +(1,1);
}
\foreach \quad in {(10, 0), (0, 10), (10,10)} {
  \foreach \x in \plaincoords
    \draw[pattern=north west lines] \quad ++\x rectangle +(1,1);
}
\coordinate (v) at (14.5,15.5);
\foreach \x in \plaincoords {
  \edef\mycmd{ (v) ++($-1*\x$) rectangle +(-1,-1);}
  \expandafter\filldraw\mycmd
}
\fill (v) circle[radius=0.3ex] node[anchor=west] {$v$};
\end{tikzpicture}
\caption{\label{fig:dblsize}The sets $v-K_1$ (filled area), $K_1$ (crosshatched area) and $K_1+\{(1,0), (0,1), (1,1)\}$ (hatched areas)}
\end{figure}
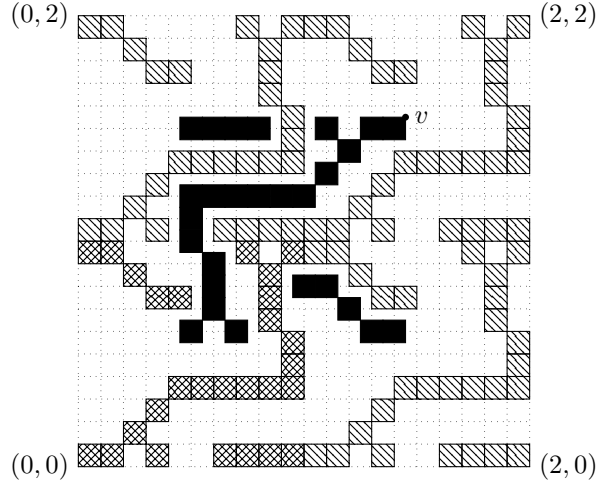

The pedantic verification of this fact is also straightforward, as 
\[K_1=\bigcup_{t\in T}\varphi_{t}([0,1]^2)=\bigcup_{(t_1, t_2)\in T} C^2_{10}(t_1, t_2)\]
is the union of \arabic{coordcount} grid squares. For example, we may note that the set
\[T^+ = T \cup (T+(10,0))\cup (T+(0,10)) \cup (T+(10,10))\subseteq \Z^2\]
satisfies that
\[K_1+\{(0,0), (1,0), (0,1), (1,1)\} = \bigcup_{(t_1, t_2)\in T^+} C^2_{10}(t_1,t_2).\]
On the other hand, it is easy to verify that if $s_1, s_2, t_1, t_2\in \Z$, then the grid square $C^2_{10}(s_1,s_2)$ intersects the square $v-C^2_{10}(t_1,t_2)$ if and only if $s_1\in \{13-t_1, 14-t_1\}$ and $s_2\in \{14-t_2,15-t_2\}$. This implies that $v-K_1$ is intersected by the grid square $C^2_{10}(s_1,s_2)$ if and only if $(s_1, s_2)\in T^-$ for the set
\[T^- = ((13,14)-T) \cup ((14,14)-T)\cup ((13,15)-T) \cup ((14,15)-T)\subseteq \Z^2.\]
Here we may mechanically check that $T^+\cap T^-=\emptyset$, which concludes the proof.
\end{proof}

\begin{claim}\label{ksmall}
The sum $K+K$ is nowhere dense in the plane.
\end{claim}

\begin{proof}
As $K+K$ is closed, it is enough to prove that $\interior(K+K)=\emptyset$, or equivalently, $\R^2\setminus (K+K)$ is dense in the plane.

First note that for $i\ge 2$ we can directly expand the recursive definition of $K_i$ to get
\[K_i = 10^{-i+1}\cdot K_1 + \sum_{j=1}^{i-1} 10^{-j}\cdot T.\]
As $T\subseteq \Z^2$ (and $10\in\Z$), this implies that
\[K\subseteq K_i \subseteq 10^{-i+1}\cdot (K_1 + \Z^2),\]
which is also trivially true for $i=1$. As $\Z^2$ is closed under addition,
\[K+K\subseteq 10^{-i+1}\cdot (K_1 + K_1 +\Z^2)\quad\text{for every positive integer }i.\]

However, \autoref{k1small} states that there exists a vector $v\in \R^d$ such that the sets $K_1+\Z^2$ and $v-K_1$ are disjoint. This means that there are no $k, k'\in K_1$ and $z\in \Z^2$ such that $k +z = v - k'$. Rearranging this (and using that the integer lattice is closed under addition) yields that $v + \Z^2$ is disjoint from $K_1 + K_1+\Z^2$, and then for every positive integer $i$
\[10^{-i+1}\cdot (K_1 + K_1 +\Z^2)\quad\text{is disjoint from}\quad 10^{-i+1}\cdot (v + \Z^2).\]

The combination of our observations proves that $K+K$ is disjoint from
\[\bigcup_{i=1}^\infty 10^{-i+1}\cdot (v + \Z^2)\]
and this is a dense set, which concludes our proof.
\end{proof}
The combination of \autoref{klarge} and \autoref{ksmall} completes the proof of \autoref{self-similar}.
\end{proof}

Although this theorem was based on the properties of the planar pattern $K_1$, it also implies the following simple result for higher dimensions:

\begin{corollary}\label{self-similar higher dimension}
For each integer $d\ge 2$, there exists a self-similar $(2,d-1,d)$-set, that is, a self-similar, compact set $K\subseteq [0,1]^d$ such that $K$ intersects each $(d-1)$-dimensional affine subspace that meets $[0, 1]^d$, but $K+K$ is nowhere dense.
\end{corollary}

\begin{proof}
Let $K_0$ denote the self-similar $(2,1,2)$-set constructed in \autoref{self-similar}. For a dimension $d\ge 2$, let us define
\[K = K_0\times [0,1]^{d-2}\subseteq[0,1]^d.\]
The iterated application of \autoref{t:product with interval} shows that $K$ is a $(2, d-1, d)$-set.

As $K_0$ is self-similar, it satisfies that \(K_0=\bigcup_{t\in T}\varphi_t(K_0)\) for some family $(\varphi_t)_{t\in T}$ of similarity transformations on the plane. For $t\in T$, let $\lambda_t$ denote the similarity ratio of $\varphi_t$ (which is clearly less than one). For each $t\in T$, consider the finite set
\[F_t = ((\lambda_t\cdot\mathbb{N})\cap[0, 1-\lambda_t]) \cup \{1-\lambda_t\},\]
which clearly satisfies that $F_t+[0,\lambda_t]=[0,1]$.

For a $t\in T$ and a vector $f=(f_3, f_4, \ldots, f_{d})$ where $f_i\in F_t$ for $2< i\le d$ let us define $\psi_{t, f}\colon \R^d\to \R^d$ as
\[\psi_{t,f}(x)=\left((\varphi_t((x_1,x_2)))_1, (\varphi_t((x_1,x_2)))_2, \lambda_t x_3 + f_3, \lambda_t x_4 + f_4,\ldots, \lambda_t x_d + f_d\right).\]
It is clear from this definition that $\psi_{t,f}$ is a similarity with ratio $\lambda_t$. The choice of $F_t$ implies that
\[\bigcup_{t \in T}\bigcup_{f\in F_t^{d-2}} \psi_{t,f}(K) = K,\]
which shows that $K$ is indeed self-similar.
\end{proof}

\begin{remark}
The so-called open set condition (see e.g.~\cite{Ma}) is satisfied by the self-similar $(2,1,2)$-set constructed in \autoref{self-similar} and the self-similar $(2,d-1,d)$-sets constructed in \autoref{self-similar higher dimension}.
\end{remark}

\begin{proof}
Recall that in the proof of \autoref{self-similar} we defined the self-similar $(2,1,2)$-set with the help of a family $(\varphi_t)_{t\in T}$ of similarity transformations. It is easy to verify that for the open unit square $(0,1)^2$ the images $\varphi_t((0,1)^2)$ ($t\in T$) are pairwise disjoint, which shows that the open set condition holds.

To prove the second claim we have to use the additional fact that the similarity ratio of $\varphi_t$ is $\lambda_t = \frac1{10}$ for each $t\in T$ (according to the definition in the proof of \autoref{self-similar}). As this common value happens to be the reciprocal of an integer, we can verify that the similarities $\psi_{t,f}$ (defined in the proof of \autoref{self-similar higher dimension}) will map the open cube $(0,1)^d$ onto disjoint subsets of itself.
\end{proof}

\section{Further remarks and open problems}
\label{sec:open}

First, we mention some facts about the possible Hausdorff dimension of the compact sets we constructed. 

\begin{remark}
  As we saw in \autoref{c:(2,k,d) <=> k >= d/2}, there exists a $(2, k, d)$-set if and only if $k \ge \frac{d}{2}$. In this case there exists even a $(2, k, d)$-set with Hausdorff dimension $d-k$ and another one with Hausdorff dimension $d$. (It is easy to see that these are the smallest and the largest possible values, since there is no $(2, k, d)$-set of Hausdorff dimension less than $d-k$, as \autoref{f:projections <-> intersections} with $m = d-k$ implies that such sets have full $(d-k)$-dimensional projections.) 
  
  To get a $(2, k, d)$-set with Hausdorff dimension $d-k$ note that a standard argument shows that the generic element of $\mathcal{K}_k([0, 1]^d)$ has Hausdorff dimension $d-k$, while \autoref{thgeneric} implies that the generic element is also a $(2, k, d)$-set.
  
  To get a $(2, k, d)$-set with Hausdorff dimension $d$, one needs go through the proof of \autoref{t:K+K construction} and notice the following. If the sequence $a_n$ tends to infinity fast enough then at each subdivision of a cube $C \in \mathcal{C}_n$ a fixed ratio of the smaller cubes is chosen with high probability. A standard argument then shows that the resulting set indeed has Hausdorff dimension $d$. 
\end{remark}

\begin{remark} 
It would also be interesting to examine what one can say about the box, packing, and topological dimensions of $(\ell, k, d)$-sets.
\end{remark} 

We now move on to open problems. So far, very little seems to be known about the case $\ell > d$. The first instance of an open question is the following.

\begin{question}
Does there exist a $(3,1,2)$-set, that is, does there exist a compact set $K\subseteq [0,1]^2$ such that $K$ intersects each line that meets $[0, 1]^2$, but $K+K+K$ is nowhere dense?
\end{question}

We cannot even rule out the following.

\begin{question}
Does there exist a set that is an $(\ell,1,2)$-set for every $\ell$, that is, does there exist a compact set $K\subseteq [0,1]^2$ such that $K$ intersects each line that meets $[0, 1]^2$, but the $\ell$-fold sumset $\ell K$ is nowhere dense for each $\ell$?
\end{question}

\begin{remark} We remark that the additive semigroup generated by a suitable translate of such a set would yield an interesting example of a meager $F_\sigma$ subsemigroup intersecting every line in the plane.
\end{remark} 
%TODO: EZ IGAZ? A CSOPORT IS JÓ? VAGY ILYENEK TRIVIN VANNAK?

\begin{remark} Also observe that if there exists a $(\ell,1,2)$-set $K$ with some $\ell \ge 4$ then $K+K$ is totally disconnected, hence has topological dimension $0$. Indeed, assume to the contrary that $C$ is a non-trivial connected component of $K+K$. First, if $C$ is not contained in a line, then \cite[Theorem~3]{Banakh} implies that $C+C$ has non-empty interior, contradicting that $K+K+K+K$ is nowhere dense. Secondly, if $C$ is a line segment, then using that $K$ intersects all lines in this particular direction that meets $[0,1]^2$, an easy category argument implies that $K+C$ has non-empty interior, contradicting even that $K+K+K$ is nowhere dense.
\end{remark}

This motivates the next question. Recall that a totally disconnected compact set with no isolated point is called a \emph{Cantor set}.   

\begin{question}
Does there exist a compact set $K\subseteq [0,1]^2$ such that $K$ intersects each line that meets $[0, 1]^2$, but $K+K$ is a Cantor set?
\end{question}

\bigskip
In the easier-looking case $\ell \le d$ the following instance is the first one left open by our results. 

\begin{question}
Does there exist a $(3, 3, 5)$-set, that is, does there exist a compact set $K\subseteq [0,1]^5$ such that $K$ intersects each $3$-dimensional affine subspace that meets $[0,1]^5$, but $K+K+K$ is nowhere dense?
\end{question}

Our last question asks if \autoref{self-similar higher dimension} can be generalized.

\begin{question}
$\exists (\ell, k, d)$-set $\implies \exists$ a self-similar $(\ell, k, d)$-set? 
That is, assume that there exists a compact set $K\subseteq [0,1]^d$ such that $K$ intersects each $k$-dimensional affine subspace that meets $[0,1]^d$, but the $\ell$-fold sumset $\ell K$ is nowhere dense, does this imply that there exists such a set $K$ that is also self-similar?
\end{question}

\subsection*{Acknowledgments}
We are indebted to Imre~Z.~Ruzsa for an illuminating conversation.

\end{document}